\documentclass[11pt,reqno]{amsart}
\usepackage[papersize={8.5in,11in}, margin=1in]{geometry}
\usepackage[utf8]{inputenc}
\usepackage{amsthm,thmtools,amssymb}
\usepackage{enumitem}
\usepackage[colorlinks=true,pagebackref]{hyperref}
\usepackage{cleveref}
\hypersetup{%
    colorlinks=true,
    linkcolor=violet,
    citecolor=violet
}
\usepackage{tikz}
\usetikzlibrary{arrows.meta,decorations.pathreplacing,decorations.markings,shapes,calc}
\usetikzlibrary{matrix,arrows,decorations.pathmorphing,backgrounds,decorations.markings,positioning}
\tikzset{
  pto/.style={->,postaction={decorate},
    decoration={
        markings,
        mark=at position 0.5 with {\arrow{|}}}
  },
}
\tikzset{labelsize/.style={font=\scriptsize}}
\numberwithin{equation}{section}
\setlist[description]{font=\normalfont}

\declaretheorem[style=plain,numberwithin=section,name=Theorem]{theorem}
\declaretheorem[style=plain,sibling=theorem,name=Lemma]{lemma}
\declaretheorem[style=plain,sibling=theorem,name=Proposition]{proposition}
\declaretheorem[style=plain,sibling=theorem,name=Corollary]{corollary}

\declaretheorem[style=definition,qed=$\blacksquare$,sibling=theorem,name=Definition]{definition}
\declaretheorem[style=definition,qed=$\blacksquare$,sibling=theorem,name=Example]{example}
\declaretheorem[style=definition,qed=$\blacksquare$,sibling=theorem,name=Remark]{remark}

\crefname{theorem}{Theorem}{Theorems}
\crefname{appendix}{Appendix}{Appendices}
\crefname{section}{Section}{Sections}
\crefname{subsection}{Subsection}{Subsections}
\crefname{definition}{Definition}{Definitions}
\crefname{example}{Example}{Examples}
\crefname{remark}{Remark}{Remarks}
\crefname{equation}{}{}
\crefname{corollary}{Corollary}{Corollaries}
\crefname{proposition}{Proposition}{Propositions}
\crefname{lemma}{Lemma}{Lemmas}

\makeatletter
\newcommand{\pto}{}
\newcommand{\pgets}{}

\DeclareRobustCommand{\pto}{\mathrel{\mathpalette\p@to@gets\to}}
\DeclareRobustCommand{\pgets}{\mathrel{\mathpalette\p@to@gets\gets}}

\newcommand{\p@to@gets}[2]{%
  \ooalign{\hidewidth$\m@th#1\mapstochar\mkern5mu$\hidewidth\cr$\m@th#1\to$\cr}%
}
\makeatother

\newcommand{\Hom}{\mathrm{Hom}}
\newcommand{\Pow}{\mathcal{P}}
\newcommand{\id}{\mathrm{id}}
\newcommand{\len}{\mathrm{len}}
\newcommand{\Conn}{\mathrm{Conn}}
\newcommand{\norm}[1]{\left\lVert#1\right\rVert_1}
\newcommand{\proj}{\mathrm{proj}}
\newcommand{\lift}{\mathrm{lift}}
\newcommand{\Aut}{\mathrm{Aut}}
\newcommand{\gr}[1]{\left\lVert #1 \right\rVert}
\newcommand{\ev}{\mathrm{ev}}

\title[Hom complexes and square-freeness]{Homotopy types of Hom complexes of graph homomorphisms whose codomains are square-free}

\author[S. Fujii]{Soichiro Fujii}
\address{
Department of Mathematics and Statistics, Faculty of Science, Masaryk University, Kotl\'a\v{r}sk\'a 2, 611 37 Brno, Czech Republic}
\email{s.fujii.math@gmail.com}

\author[K. Kimura]{Kei Kimura}
\address{
Faculty of Information Science and Electrical Engineering, Kyushu University \\ 744, Motooka, Nishi-ku, Fukuoka, 819-0395, Japan}
\email{kkimura@inf.kyushu-u.ac.jp}

\author[Y. Nozaki]{Yuta Nozaki}
\address{
Faculty of Environment and Information Sciences, Yokohama National University \\
79-7 Tokiwadai, Hodogaya-ku, Yokohama, 240-8501 \\
Japan\vspace{-0.6em}}
\address{
WPI-SKCM$^2$, Hiroshima University \\
1-3-1 Kagamiyama, Higashi-Hiroshima, Hiroshima, 739-8526 \\
Japan}
\email{nozaki-yuta-vn@ynu.ac.jp}

\subjclass[2020]{Primary 55U05, 05C15, Secondary 55P15, 06A15}

\keywords{Hom complex, square-free graph, homotopy type, poset topology}

\begin{document}
\begin{abstract}
Given finite simple graphs $G$ and $H$, the Hom complex $\mathrm{Hom}(G,H)$ is a polyhedral complex having the graph homomorphisms $G\to H$ as the vertices. We determine the homotopy type of each connected component of $\mathrm{Hom}(G,H)$ when $H$ is square-free, meaning that it does not contain the $4$-cycle graph $C_4$ as a subgraph. Specifically, for a connected $G$ and a square-free $H$, we show that each connected component of $\mathrm{Hom}(G,H)$ is homotopy equivalent to a wedge sum of circles. We further show that, given any graph homomorphism $f\colon G\to H$ to a square-free $H$, one can determine the homotopy type of the connected component of $\mathrm{Hom}(G,H)$ containing $f$ algorithmically. 
\end{abstract}

\maketitle
\setcounter{tocdepth}{1}
\tableofcontents

\section{Introduction}\label{sec:intro}
The \emph{Hom complex} $\Hom(G,H)$ is a certain polyhedral complex associated with a pair of finite simple graphs $G$ and $H$. Hom complexes have been used in the algebro-topological approach to the graph coloring problem; see e.g.~\cite{Lov78,Koz08}.
Although it is known that the homotopy type of a Hom complex can be quite complicated in general (see e.g.~\cite[Theorem~1.10]{DoSc12}), the homotopy type of (each connected component of) a Hom complex has been determined in certain special cases. For example, Babson and Kozlov \cite{BaKo06} showed that the Hom complex $\Hom(K_n,K_k)$ between complete graphs is homotopy equivalent to a wedge sum of $(k-n)$-spheres when $n \leq k$. 
We have shown in a recent paper \cite{cycle} that if $G$ is a finite connected graph and $C_k$ is the $k$-cycle graph with $k\geq 3$, then each connected component of $\Hom(G,C_k)$ is homotopy equivalent to a point or a circle, generalizing an earlier result by \v Cuki\'c and Kozlov \cite{CuKo06}.

In this paper, we determine the homotopy type of each connected component of $\Hom(G,H)$ when $H$ is square-free, meaning that $H$ does not contain the $4$-cycle graph $C_4$ as a (not necessarily induced) subgraph. 
To state our result concisely, we assume that $G$ and $H$ are connected and have at least two vertices; it is straightforward to deduce the results for a general $G$ and a general square-free $H$ from this case, as we will see in \cref{subsec:simplification}.
The following is the main theorem of this paper.

\begin{theorem}\label{thm:main}
    Let $G$ and $H$ be finite connected graphs with at least two vertices, with $H$ square-free. Then each connected component of $\Hom(G,H)$ is homotopy equivalent to either
    \begin{enumerate}[label=\emph{({\arabic*})}]
        \item a point,
        \item a circle, or 
        \item (the geometric realization of) a connected component of the graph $H\times K_2$.
    \end{enumerate}
    In particular, each connected component of $\Hom(G,H)$ is homotopy equivalent to a wedge sum of circles.
\end{theorem}

Here, $H\times K_2$ is the (categorical) product (recalled in \cref{subsec:graphs}) of $H$ and the complete graph $K_2$ with two vertices (i.e., an edge). 
It is easy to describe $H\times K_2$:
\begin{itemize}
    \item If $H$ is bipartite, then we have $H\times K_2\cong H \sqcup H$, where $\sqcup$ denotes a disjoint union of graphs. Hence $H\times K_2$ contains two connected components, both of which are isomorphic to $H$.
    \item If $H$ is not bipartite, then $H\times K_2$ is connected. It is a connected double cover of $H$ \cite{Waller-double-cover}. 
\end{itemize}
In particular, for the cycle graph $C_k$ ($k\geq 3$), we have 
\[
C_k\times K_2\cong \begin{cases}
    C_k \sqcup C_k &\text{if $k$ is even, and}\\
    C_{2k}  &\text{if $k$ is odd,}
\end{cases}
\]
and hence we obtain the abovementioned result of \cite{cycle} (with $k\neq 4$) as a special case.
Furthermore, we can omit the case (3) in \cref{thm:main} when $G$ is not bipartite. In fact, we can refine \cref{thm:main} as follows.
\begin{theorem}\label{prop:main-cases}
    Let $G$ and $H$ be finite connected graphs with at least two vertices, with $H$ square-free.
    \begin{itemize}
        \item If both $G$ and $H$ are bipartite, then $\Hom(G,H)$ is homotopy equivalent to the disjoint union of points, circles, and exactly two copies of (the geometric realization of) $H$.
        \item If $G$ is bipartite and $H$ is not, then $\Hom(G,H)$ is homotopy equivalent to the disjoint union of points, circles, and exactly one copy of (the geometric realization of) $H\times K_2$.
        \item If $H$ is bipartite and $G$ is not, then $\Hom(G,H)$ is empty. 
        \item If neither $G$ nor $H$ is bipartite, then $\Hom(G,H)$ is homotopy equivalent to the disjoint union of points and circles.
    \end{itemize}
\end{theorem}
(Notice that here we are \emph{not} claiming, e.g.~in the first case, that $\Hom(G,H)$ contains exactly two connected components which are homotopy equivalent to $H$. Indeed, the geometric realization of $H$ can be homotopy equivalent to either a point or a circle.)
Moreover, we can use an algorithm given in \cite{Wro20} to determine the homotopy type of the connected component of $\Hom(G,H)$ containing a given graph homomorphism $f\colon G\to H$ in polynomial time (see \cref{rmk:algo-for-Lambda}).

We now turn to an outline of the proof of our main theorem. The basic strategy is similar to our previous work \cite{cycle}. Given a graph homomorphism $f\colon G\to H$ between graphs satisfying the assumption of \cref{thm:main}, we explicitly describe the universal cover $p_f\colon E_f\to B_f$ of the connected component $B_f$ of $\Hom(G,H)$ containing $f$. This allows us (after some work) to prove that $E_f$ is contractible. By a standard result in algebraic topology (see e.g.~\cite[Appendix~1.B]{Hat02}), this implies that $B_f$ is an \emph{aspherical space}, i.e., that we have $\pi_n(B_f)\cong \{1\}$ for all natural numbers $n\neq 1$. 
More precisely, $B_f$ is the Eilenberg--MacLane space $K\bigl(\pi_1(B_f),1\bigr)$. In particular, the homotopy type of $B_f$ is completely determined by its fundamental group $\pi_1(B_f)$.
Another standard result in algebraic topology (see e.g.~\cite[Proposition~1.39]{Hat02}) says that $\pi_1(B_f)$ is isomorphic to the automorphism group $\Aut(p_f)$ of the universal cover $p_f\colon E_f\to B_f$. Hence we can once again exploit our explicit description of the universal cover to determine the latter group.
It turns out that we can apply results of \cite{Wro20} for this.

Thus the key step is the explicit construction of the universal cover $p_f\colon E_f\to B_f$. To this end, we introduce the \emph{graph $\Pi H$ of reduced walks} in a graph $H$. We will construct $E_f$ as a suitable subcomplex of the Hom complex $\Hom(G,\Pi H)$ with $\Pi H$ as the codomain. 

Throughout, we employ the technique of \emph{poset topology} (briefly recalled in \cref{subsec:poset-topology}) and give combinatorial arguments involving graphs and posets, rather than topological arguments.

\subsection*{Outline of this paper}
After recalling basic notions and fixing our notation in \cref{sec:prelim}, we introduce the graph $\Pi H$ of reduced walks in a graph $H$ in \cref{sec:PiH}.
The remaining three sections (\cref{sec:covering,sec:contractibility,sec:fundamental-group}) are devoted to the proofs of the main theorem and its refinement (\cref{thm:main,prop:main-cases}).
Thus, given a homomorphism $f\colon G\to H$ between graphs satisfying the assumption of \cref{thm:main}, we want to determine the homotopy type of the connected component $B_f$ of $\Hom(G,H)$ containing $f$. 
\cref{sec:covering} shows that a certain continuous map $p_f\colon E_f\to B_f$ defined via $\Pi H$ is a covering map. \cref{sec:contractibility} shows that $E_f$ is contractible. This in particular implies that $E_f$ is simply-connected, i.e., $p_f\colon E_f\to B_f$ is the universal cover of $B_f$. Finally, in \cref{sec:fundamental-group}, we determine the automorphism group $\Aut(p_f)$ of the universal cover $p_f\colon E_f\to B_f$, thus leading to the proofs of \cref{thm:main,prop:main-cases}. 
In \cref{apx:comparing}, we discuss a different approach to construct the universal cover $p_f\colon E_f\to B_f$.

\subsection*{Related work}
As mentioned above, this paper generalizes \cite{cycle}, and an explicit description of the universal cover is the key step in both \cite{cycle} and in this paper.
A novel contribution in this paper includes the introduction of the graph $\Pi H$.
The details of the proofs are also different.
Thanks to the simpler situation, in \cite{cycle} we were able to construct the universal cover and prove its contractibility by a direct topological method;
here we use results of poset topology and give a more combinatorial proof. 
Another difference, to the advantage of \cite{cycle}, is that the algorithm to determine the homotopy type of a given connected component of the Hom complex given in \cite[Theorem~5.1]{cycle} is considerably simpler than the corresponding algorithm (essentially that of \cite{Wro20}) in this paper. 

The work \cite{Wro20} of Wrochna is closely related to this paper.
Although \cite{Wro20} does not explicitly deal with Hom complexes and instead focuses on the $H$-recoloring problem (which is essentially the reachability problem in a Hom complex), many crucial observations in this paper are already present in \cite{Wro20}, as we point out in due course. 
Indeed, as we will see in \cref{rmk:algo-for-Lambda}, we can even use the main algorithm of \cite{Wro20} to determine the homotopy type of $B_f$.
Incidentally, the paper \cite{Wro20} concludes with the following sentence.  
\begin{quote}
    Finally, it could be interesting to explore the implications of the square-free property for the whole Hom complex.
\end{quote}
\cref{thm:main,prop:main-cases} provide an answer to this.

Hom complexes of the form $\Hom(K_2,H)$ have been studied widely. For any finite graph $H$, $\Hom(K_2,H)$ is known to be homotopy equivalent to the \emph{neighborhood complex} $\mathrm{N}(H)$ of $H$ introduced in \cite{Lov78} (see \cite[Proposition~4.2]{BaKo06}), as well as to the \emph{box complex} $\mathrm{B}(H)$ of $H$ introduced in \cite{MZ-box} (see \cite[Remark~5.4]{Cso07}).
For a connected square-free $H$, Kamibeppu \cite{Kamibeppu} showed that $\Hom(K_2,H)$ is homotopy equivalent to (the geometric realization of) $H\times K_2$. (This can also be seen from our proof of \cref{prop:main-cases}.)
The fact that the homotopy type of $\Hom(K_2,H)$ becomes simple when $H$ is square-free had been anticipated in \cite[Section~4]{MZ-box}.
In fact, $\Hom(K_2,H)$ admits a $\mathbb{Z}/2\mathbb{Z}$-action induced from the natural $\mathbb{Z}/2\mathbb{Z}$-action on $K_2$, and \cite{Kamibeppu} determines the $\mathbb{Z}/2\mathbb{Z}$-homotopy type of $\Hom(K_2,H)$, which is a stronger result than the determination of its ordinary homotopy type.

Matsushita~\cite{Mat25} independently proved the main theorem of this paper, after he had a few discussions with the third-named author of this paper about the contents of \cite{cycle} and the statement (but not the proof) of our main theorem in this paper. 
Interestingly, his proof involves a \emph{different} explicit construction of the universal cover of $B_f$ based on his earlier work \cite{Mat17JMSUT}. 
We describe the relationship between the two approaches in \cref{apx:comparing}; see in particular \cref{ex:cover-approach}.

\subsection*{Remarks on combinatorial applications}
As mentioned at the beginning of this section, Hom complexes have been introduced in the context of the graph coloring problem.
Our main theorem suggests that $\Hom(G,H)$ would be too simple to apply to the graph coloring problem when $H$ is square-free, generalizing the aforementioned remark of \cite[Section~4]{MZ-box} from the case where $G=K_2$ to arbitrary $G$. 
Whereas this observation could be useful in preventing fruitless attempts, the reader might well like to see a more positive application.

We claim that, from the viewpoint of the $H$-recoloring problem \cite{Wro20}, the homotopy type of $\Hom(G,H)$ (or even just the fundamental groupoid of $\Hom(G,H)$) provides deeper understanding to the structure of all solutions to a given instance of the $H$-recoloring problem. Indeed, such a solution gives rise to a path in $\Hom(G,H)$, and the existence of an (endpoint-preserving) homotopy between such paths in $\Hom(G,H)$ can be regarded as giving a reasonable notion of equivalence between the solutions. 
Thus the fundamental groupoid of $\Hom(G,H)$ reveals how many ``essentially different'' solutions a given instance of the $H$-recoloring problem admits. As our discussion in \cref{sec:fundamental-group} shows, for a square-free $H$, Wrochna's notion of \emph{realizable walk} \cite[Section~3]{Wro20} gives an answer to this question: two solutions are equivalent in the above sense if and only if the associated realizable walks coincide. 

\subsection*{Acknowledgments}
We thank Takahiro Matsushita for the several fruitful discussions.
In particular, we note that
\begin{itemize}
    \item the use of the graph $H\times K_2$ in the statement of \cref{thm:main}, which replaces our earlier formulation in terms of the rank of the (free) fundamental group $\pi_1(B_f)$ (see \cref{eqn:rank-of-fundamental-group}), is suggested by him,
    \item the proof of contractibility of $E_f$ given in \cref{sec:contractibility}, which improves our previous proof, is based on his proof, and
    \item the contents of \cref{apx:comparing} grew out of discussions with him.
\end{itemize}

We are grateful to the anonymous referees for their helpful comments. 

The first-named author is supported by the Grant Agency of the Czech Republic under the grant 22-02964S.
The second-named author is supported by JST ERATO Grant Number JPMJER2301 and JSPS KAKENHI Grant Number JP21K17700, Japan.
The third-named author is supported by JSPS KAKENHI Grant Numbers JP23K12974 and JP24H00686, Japan.

\section{Preliminaries}\label{sec:prelim}
In \cref{subsec:graphs,subsec:graph-hom,subsec:poset-topology}, we recall basic notions used in this paper and fix our notation. 
In \cref{subsec:simplification}, we justify the simplifying assumption in \cref{thm:main}.
We denote by $\mathbb{N}$ the set $\{\,0,1,2,\dots\,\}$
of natural numbers.

\subsection{Graphs}\label{subsec:graphs}
By a \emph{graph} $G$, we mean a pair $G=\bigl(V(G),E(G)\bigr)$, where $V(G)$ is a set and $E(G)$ is a set of subsets of $V(G)$ of cardinality two; that is, our graphs are undirected and have no loops nor multiple edges. A graph $G$ is \emph{finite} if $V(G)$ is a finite set and \emph{infinite} otherwise. 
In this paper, graphs are \emph{not} assumed to be finite unless explicitly stated so. 
Elements of $V(G)$ are called \emph{vertices} of $G$ and elements of $E(G)$ \emph{edges} of $G$. An edge $\{u,v\}$ is also denoted by $uv$.
Two vertices $u$ and $v$ of $G$ are \emph{adjacent} if $uv\in E(G)$.
The \emph{neighborhood} $N_G(u)$ of $u\in V(G)$ is the set of all vertices of $G$ adjacent to $u$.
A vertex of $G$ is \emph{isolated} if it is not adjacent to any vertex of $G$.

A graph $H$ is \emph{square-free} if it does not contain the $4$-cycle graph $C_4$ as a (not necessarily induced) subgraph. Equivalently, $H$ is square-free if and only if we have 
\begin{equation}\label{eqn:sq-free}
    |N_H(x)\cap N_H(y)|\leq 1
\end{equation}
for each pair of distinct vertices  $x,y\in V(H)$ \cite[Section~1]{Wro20}.

Given graphs $G$ and $H$, their (categorical) \emph{product} $G\times H$ is defined by $V(G\times H)=V(G)\times V(H)$ and 
\[
E(G\times H)=\bigl\{ \{(u,x),(v,y)\} \,\big\vert\, uv\in E(G)\text{ and } xy\in E(H) \bigr\}.
\]

\subsection{Graph homomorphisms and set-valued homomorphisms}\label{subsec:graph-hom}
Let $G$ and $H$ be graphs.

A \emph{graph homomorphism} from $G$ to $H$ is a function $f\colon V(G)\to V(H)$ such that, for each $u,v\in V(G)$, if $u$ and $v$ are adjacent in $G$, then so are $f(u)$ and $f(v)$ in $H$.
We write $f\colon G\to H$ to mean that $f$ is a graph homomorphism from $G$ to $H$.

A graph homomorphism $f\colon G\to H$ is said to be a \emph{covering map of graphs}, if for each $u\in V(G)$ and each $x\in V(H)$ such that $f(u)$ and $x$ are adjacent in $H$, there exists a unique $v\in V(G)$ such that $u$ and $v$ are adjacent in $G$ and $f(v)=x$; in other words, a graph homomorphism $f\colon G\to H$ is a covering map if and only if $f|_{N_G(u)}\colon N_{G}(u)\to N_H\bigl(f(u)\bigr)$ is a bijection for each $u\in V(G)$.

A \emph{set-valued homomorphism} from $G$ to $H$ is a function $\varphi\colon V(G)\to \Pow\bigl(V(H)\bigr)\setminus\{\emptyset\}$, where $\Pow\bigl(V(H)\bigr)$ is the set of all subsets of $V(H)$, such that for each $uv\in E(G)$, each $x\in \varphi(u)$, and each $y\in \varphi(v)$, we have $xy\in E(H)$. We write $\varphi\colon G\pto H$ to mean that $\varphi$ is a set-valued homomorphism from $G$ to $H$.
Notice that any graph homomorphism $f\colon G\to H$ can be regarded as a set-valued homomorphism $f\colon G\pto H$ mapping each $u\in V(G)$ to the singleton $\{f(u)\}$.
For set-valued homomorphisms $\varphi,\psi\colon G\pto H$, we write $\varphi \leq \psi$ if we have $\varphi(u)\subseteq \psi(u)$ for each $u\in V(G)$. 
This defines a partial order; the resulting poset of all set-valued homomorphisms from $G$ to $H$ is denoted by $\Hom(G,H)$. 
The \emph{Hom complex}, which is the main subject of this paper, is the topological space $\gr{\Hom(G,H)}$ naturally associated with the poset $\Hom(G,H)$ by the process explained in \cref{subsec:poset-topology}. As explained there, we can investigate the topological properties of $\gr{\Hom(G,H)}$ by means of combinatorial arguments on the poset $\Hom(G,H)$.
See \cite[Chapter~18]{Koz08} or \cite[Section~1]{BaKo06} for examples of Hom complexes.

Two graph homomorphisms $f,g\colon G\to H$ are \emph{adjacent} if there exists $u\in V(G)$ such that 
\begin{itemize}
    \item $f(u)\neq g(u)$ and 
    \item for all $v\in V(G)\setminus\{u\}$, $f(v)=g(v)$.
\end{itemize}
If this is the case, then the function $\varphi\colon V(G)\to \Pow\bigl(V(H)\bigr)\setminus \{\emptyset\}$, mapping $v\in V(G)$ to $\varphi(v)=\{f(v),g(v)\}$ (which is a singleton unless $v=u$), is a set-valued homomorphism $\varphi\colon G\pto H$. We have $f\leq \varphi$ and $g\leq \varphi$ in $\Hom(G,H)$. 

If $K$ is another graph and $k\colon H\to K$ is a graph homomorphism, then we have a monotone map 
\[
\Hom(G,k)\colon \Hom(G,H)\to \Hom (G,K)
\]
mapping $\varphi\colon G\pto H$ to $k\varphi\colon G\pto K$ given by $(k\varphi)(u)=k\bigl(\varphi(u)\bigr)$ for each $u\in V(G)$.

\subsection{Poset topology}\label{subsec:poset-topology}
In this paper, whenever we refer to a \emph{subposet} of a poset $P=(P,\leq_P)$, we mean a subset $S$ of $P$ equipped with the order relation $\leq_S$ defined as $p\leq_Sq$ if and only if $p\leq_Pq$ for each $p,q\in S$.

For any poset $P$, the \emph{order complex} $\Delta(P)$ of $P$ is the abstract simplicial complex whose vertex set is $P$ and whose simplices are the finite chains in $P$. Its \emph{geometric realization} $\gr{\Delta(P)}$ is the topological space obtained by gluing geometric simplices according to $\Delta(P)$; see e.g.~\cite[Section~9]{Bjorner-topology} for details. 
The topological space $\gr{\Delta(P)}$ is called the \emph{classifying space} of $P$.

Any monotone map $\alpha\colon P\to Q$ between posets induces a simplicial map $\Delta(\alpha)\colon \Delta(P)\to \Delta(Q)$, and hence a continuous map $\gr{\Delta(\alpha)}\colon \gr{\Delta(P)}\to\gr{\Delta(Q)}$ between the classifying spaces. 
In what follows, we simply denote $\gr{\Delta(P)}$ by $\gr{P}$ and $\gr{\Delta(\alpha)}$ by $\gr{\alpha}$.

We now recall that the notion of ``connected component'' is compatible with this process of turning a poset $P$ into its classifying space $\gr{P}$. 
Let $P$ be a poset. Two elements $p$ and $q$ of $P$ are \emph{connected} if $p$ and $q$ are related by the equivalence relation generated by the order relation $\leq$ on $P$. In other words, $p$ and $q$ are connected if and only if there exist $n\in\mathbb{N}$ and $r_1,r_2,\dots, r_{2n}\in P$ such that $p\leq r_1\geq r_2\leq \dots \geq r_{2n}\leq q$. For any $p\in P$, we denote by $\Conn(P,p)$ the subposet of $P$ consisting of all elements of $P$ connected to $p$. 
On the other hand, every element $p$ of $P$ gives rise to a $0$-simplex $\{p\}$ of $\Delta(P)$, and hence a point of $\gr{\Delta(P)}=\gr{P}$, which we denote by $x_p\in\gr{P}$. Two elements $p$ and $q$ of $P$ are connected in the above sense if and only if the two points $x_p$ and $x_q$ are (path-)connected in $\gr{P}$. 
Indeed, writing the connected component of a topological space $X$ containing a point $x\in X$ as $\Conn(X,x)$, we have $\Conn(\gr{P},x_p)\cong\gr{\Conn(P,p)}$ for any poset $P$ and any $p\in P$.

Now we can state our main goal in this paper more precisely. Given a graph homomorphism $f\colon G\to H$ between graphs satisfying the assumption of \cref{thm:main}, we will determine the homotopy type of the connected component of the Hom complex $\gr{\Hom(G,H)}$ containing (the point corresponding to) $f$, i.e., the classifying space of the poset $\Conn\bigl(\Hom(G,H),f\bigr)$. The poset $\Conn\bigl(\Hom(G,H),f\bigr)$ is denoted by $B_f$.

Various results relating properties of a poset $P$ (resp.~a monotone map $\alpha\colon P\to Q$) and those of the topological space $\gr{P}$ (resp.~the continuous map $\gr{\alpha}\colon \gr{P}\to \gr{Q}$) are known; see e.g.~\cite{Bjorner-topology,Wachs-poset}. We will recall some of them when needed. 
Due to the intimate relationship between $P$ and $\gr{P}$, it is convenient to blur the distinction. Thus we will call the poset $\Hom(G,H)$ a \emph{Hom complex}. Also, we say that a poset $P$ is \emph{homotopy equivalent to} a topological space $X$ if $\gr{P}$ and $X$ are homotopy equivalent; a poset $P$ is \emph{contractible} if $\gr{P}$ is;
a subposet $S$ of a poset $P$ is a \emph{strong deformation retract} if the subspace $\gr{S}$ of $\gr{P}$ is; and so on.

\subsection{Simplification}\label{subsec:simplification}
In this subsection, we explain why our assumption in \cref{thm:main}, that the finite graphs $G$ and $H$ are connected and have at least two vertices, is harmless.

First observe that, since we have $\Hom(G\sqcup G',H)\cong \Hom(G,H)\times \Hom(G',H)$ in general \cite[Section~2.4]{BaKo06}, where $G\sqcup G'$ denotes the disjoint union of two graphs $G$ and $G'$, we may restrict our attention to the case where $G$ is connected. Moreover, it is clear that $\Hom(G,H)$ is contractible when $G$ consists of only one vertex (and $H$ is nonempty). Hence we may moreover assume that $G$ has at least two vertices. 
We can also assume that $H$ is connected. Indeed, let $f\colon G\to H$ be a graph homomorphism, where $G$ is a connected graph with at least two vertices. Consider the factorization
\[
\begin{tikzpicture}[baseline=-\the\dimexpr\fontdimen22\textfont2\relax ]
      \node(0) at (0,0.5) {$G$};
      \node(1) at (1.5,-0.5) {$H'$};
      \node(2) at (3,0.5) {$H$};
      \draw [->] (0) to node[auto,swap,labelsize] {$f'$} (1);
      \draw [->] (1) to node[auto,swap,labelsize] {inclusion} (2);
      \draw [->] (0) to node[auto,labelsize] {$f$} (2);
\end{tikzpicture}
\]
where $H'$ is the connected component of $H$ containing the image of $f$. Then it is easy to see that the connected component of $\Hom(G,H)$ containing $f$ is canonically order isomorphic to the connected component of $\Hom(G,H')$ containing $f'$. 
Finally, if $H$ consists of only one vertex and $G$ is a connected graph with at least two vertices, then (since we only consider simple graphs) there cannot be any homomorphism $G\to H$, and hence $\Hom(G,H)$ is empty. 
Thus in \cref{sec:covering,sec:contractibility,sec:fundamental-group}, we will mainly work under the following assumption: 
\begin{quote}
    $G$ and $H$ are finite connected graphs with at least two vertices, with $H$ square-free,
\end{quote}
which we refer to as ``the assumption of \cref{thm:main}.''
Incidentally, this assumption was also adopted in \cite{Wro20}.

\section{The graph $\Pi H$ of reduced walks in a graph $H$}\label{sec:PiH}
As explained in \cref{sec:intro}, the key step in the proof of our main theorem is an explicit description of the universal cover of $B_f$, the connected component of the Hom complex $\Hom(G,H)$ containing the given homomorphism $f\colon G\to H$. 
We will construct the universal cover by first building the graph $\Pi H$ of \emph{reduced walks} in $H$.
The aim of this section is to define $\Pi H$ and establish its properties. 

\subsection{Walks and reduced walks}
\label{subsec:walk}
We begin with a review of walks and reduced walks in a graph (see e.g.~\cite[Section~2]{Wro20} and \cite[Sections~3.6 and 3.7]{Spanier}).
Let $H$ be a graph. 

A \emph{walk} in $H$ is a sequence $\xi=(x_0,x_1,\dots,x_\ell)$ of vertices of $H$ with $\ell\geq 0$, such that  $x_{i}x_{i+1}\in E(H)$ for each $0\leq i\leq \ell-1$. 
The natural number $\ell$ is called the \emph{length} of $\xi$, and is denoted by $\len(\xi)$. 
The vertex $x_0$ (resp.~$x_\ell$) is called the \emph{source} (resp.~\emph{target}) of $\xi$, and is denoted by $s(\xi)$ (resp.~$t(\xi)$).

Given $y,z\in V(H)$, a \emph{walk from $y$ to $z$} is a walk $\xi$ in $H$ such that $s(\xi)=y$ and $t(\xi)=z$.
A walk $\xi$ in $H$ with $s(\xi)=t(\xi)$ is called a \emph{closed walk} in $H$ (on $s(\xi)$).

Given walks $\xi=(x_0,x_1,\dots,x_\ell)$ and $\eta=(y_0,y_1,\dots, y_{\ell'})$ in $H$ such that $t(\xi)=s(\eta)$ (i.e., $x_\ell=y_0$), their \emph{concatenation} $\xi\bullet\eta$ is the walk $(x_0,x_1,\dots, x_{\ell}, y_1,\dots, y_{\ell'})$. 
The \emph{inverse} of a walk $\xi=(x_0,x_1,\dots,x_\ell)$ is the walk $\xi^{-1}=(x_\ell,\dots, x_1,x_0)$.

A walk $(x_0,\dots,x_\ell)$ in $H$ is \emph{reduced} if we have $x_i\neq x_{i+2}$ for each $0\leq i\leq \ell-2$. 
If a walk $\xi=(x_0,\dots,x_\ell)$ is not reduced, say because we have $x_i=x_{i+2}$ for some $0\leq i\leq \ell-2$, then the walk 
\[
(x_0,\dots, x_{i-1},x_i,x_{i+3},\dots, x_\ell)
\]
is said to be obtained from $\xi$ by \emph{elementary reduction}. 
It is not hard to see that any walk $\xi$ can be turned into a unique reduced walk $\overline \xi$ by repeated application of elementary reduction.
For reduced walks $\xi$ and $\eta$ in $H$ with $t(\xi)=s(\eta)$, their \emph{product} $\xi\cdot\eta$ is the reduced walk $\overline{\xi\bullet\eta}$. For example, for any reduced walk $\xi$, $\xi\cdot \xi^{-1}=\overline{\xi \bullet\xi^{-1}}$ is the walk $\bigl(s(\xi)\bigr)$ of length $0$.

Let $G$ and $H$ be graphs and $f\colon G\to H$ a graph homomorphism.
For any walk $\omega=(w_0,\dots, w_\ell)$ in $G$, we have a walk $\bigl(f(w_0),\dots, f(w_\ell)\bigr)$ in $H$, which we call $f(\omega)$.
Notice that we have $f(\omega\bullet\omega')=f(\omega)\bullet f(\omega')$ whenever the concatenation $\omega\bullet\omega'$ is defined, and $f(\omega^{-1})=f(\omega)^{-1}$.
For a reduced walk $\omega$ in $G$, the walk $f(\omega)$ in $H$ may fail to be reduced in general. However, we have the following.

\begin{proposition}\label{prop:covering-preserves-reduced-walks}
    Let $G$ and $H$ be graphs and $f\colon G\to H$ a covering map of graphs. Then for any reduced walk $\omega$ in $G$, the walk $f(\omega)$ in $H$ is reduced.
\end{proposition}
\begin{proof}
    Let $\omega=(w_0,\dots, w_\ell)$ be a reduced walk in $G$. For each $0\leq i\leq \ell-2$, since the function $f|_{N_G(w_{i+1})}\colon N_G(w_{i+1})\to N_H\bigl(f(w_{i+1})\bigr)$ is injective, $w_i\neq w_{i+2}$ implies $f(w_i)\neq f(w_{i+2})$. 
\end{proof}

\subsection{The graph $\Pi H$}
\label{subsec:PiH}
We now define the graph $\Pi H$ of reduced walks in a graph $H$. 

\begin{definition}
\label{def:adjacency-reduced-walks}
    Let $H$ be a graph. Two reduced walks $\xi$ and $\eta$ in $H$ are said to be \emph{adjacent} if
    \begin{itemize}
        \item $s(\xi)$ and $s(\eta)$ are adjacent in $H$,
        \item $t(\xi)$ and $t(\eta)$ are adjacent in $H$, and 
        \item $\xi=\bigl(s(\xi),s(\eta)\bigr)\cdot \eta\cdot \bigl(t(\eta),t(\xi)\bigr)$.
    \end{itemize}
    Notice that if the first two conditions are satisfied, then the final condition is equivalent to $\bigl(s(\eta),s(\xi)\bigr)\cdot \xi\cdot \bigl(t(\xi),t(\eta)\bigr)=\eta$. Thus the adjacency relation is symmetric.

    The graph $\Pi H$ is defined as follows. The set $V(\Pi H)$ of vertices of $\Pi H$ is the set of all reduced walks in $H$, and two reduced walks in $H$ are adjacent in the graph $\Pi H$ if and only if they are adjacent in above the sense.
    Notice that the graph $\Pi H$ may be infinite even when $H$ is a finite graph.
\end{definition}

\begin{remark}
    As already mentioned, the graph $\Pi H$ will be used in our construction of the universal cover of each connected component of $\Hom(G,H)$, where the graphs $G$ and $H$ satisfy the assumption of \cref{thm:main}. 
    Whereas it seems difficult to fully justify the above definition of $\Pi H$ from this viewpoint at an early stage of development
    (we believe that justifications will emerge from various properties of $\Pi H$ established in \cref{subsec:PiH,subsec:algebraic-str-HomGPiH,subsec:relation-to-Wrochna} as well as in \cref{subsec:properties-of-Pi}), here we motivate its definition by means of the notion of \emph{$S$-walk} introduced in \cite[Section~3]{Wro20}.
    Suppose that $G$ and $H$ are graphs satisfying the assumption of \cref{thm:main}, and that we have a sequence $S=(f_0,f_1,\dots,f_n)$ of graph homomorphisms from $G$ to $H$ in which $f_i$ and $f_{i+1}$ are adjacent (see \cref{subsec:graph-hom}) for each $0\leq i\leq n-1$.
    Then, for each vertex $v$ of $G$, the associated \emph{$S$-walk} $S(v)$ is the walk $\bigl(f_0(v),x_1,f_{k_1}(v),x_2,\dots, x_{n(v)},f_{k_{n(v)}}(v)\bigr)$ in $H$, where
    \begin{itemize}
        \item $k_1<\dots<k_{n(v)}$ are all integers $i \in \{1,2,\dots, n\}$ with $f_{i-1}(v)\neq f_i(v)$, and
        \item for each $1\leq j\leq n(v)$, $x_j$ is the unique element of $N_H\bigl(f_{k_j-1}(v)\bigr)\cap N_H\bigl(f_{k_j}(v)\bigr)$ (see \cref{eqn:sq-free}), or equivalently, is the vertex of $H$ satisfying $f_{k_j-1}(w)=x_j=f_{k_j}(w)$ for any vertex $w$ of $G$ adjacent to $v$.
    \end{itemize}
    The above definition of the graph $\Pi H$ is such that the assignment $v\mapsto \overline{S(v)}$ becomes a graph homomorphism $G\to \Pi H$.
    In this regard, we also remark that observations somewhat related to the definition of adjacency in \cref{def:adjacency-reduced-walks} can be found in the proof of \cite[Theorem~6.1]{Wro20}. 
\end{remark}

Intuitively, two reduced walks in $H$ are adjacent in $\Pi H$ if they ``almost coincide.'' This can be made formal as follows.
\begin{proposition}
\label{prop:adjacency-cases}
    Let $H$ be a graph and $\xi=(x_0,\dots,x_\ell)$ and $\eta=(y_0,\dots, y_{\ell'})$ be reduced walks in $H$. Then $\xi$ and $\eta$ are adjacent if and only if one of the following conditions is satisfied.
\begin{itemize}
    \item[\emph{(A1)}] We have $\ell+2=\ell'$ and $x_{i}=y_{i+1}$ for each $0\leq i\leq \ell$.
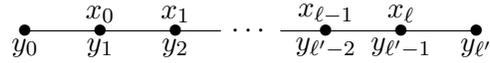
    \begin{equation*}
\begin{tikzpicture}[baseline=-\the\dimexpr\fontdimen22\textfont2\relax ]
      \node(0) at (0,0) {$\bullet$};
      \node(1) at (1,0) {$\bullet$};
      \node(2) at (2,0) {$\bullet$};
      \node(3) at (3,0) {$\cdots$};
      \node(4) at (4,0) {$\bullet$};
      \node(5) at (5,0) {$\bullet$}; 
      \node(6) at (6,0) {$\bullet$}; 
      \draw (0,0) to (3);
      \draw (3) to (6,0);
      \node at (1.north) {$x_0$}; 
      \node at (2.north) {$x_1$}; 
      \node at (4.north) {$x_{\ell-1}$}; 
      \node at (5.north) {$x_{\ell}$}; 
      \node at (0.south) {$y_0$}; 
      \node at (1.south) {$y_1$}; 
      \node at (2.south) {$y_2$}; 
      \node at (4.south) {$y_{\ell'-2}$}; 
      \node at (5.south) {$y_{\ell'-1}$}; 
      \node at (6.south) {$y_{\ell'}$}; 
\end{tikzpicture}
\end{equation*}
    \item[\emph{(A2)}] We have $\ell=\ell'>0$ and $x_{i+1}=y_i$ for each $0\leq i\leq \ell-1$.
    
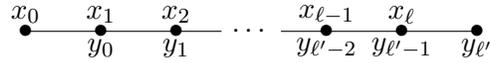
\begin{equation*}
\begin{tikzpicture}[baseline=-\the\dimexpr\fontdimen22\textfont2\relax ]
      \node(0) at (0,0) {$\bullet$};
      \node(1) at (1,0) {$\bullet$};
      \node(2) at (2,0) {$\bullet$};
      \node(3) at (3,0) {$\cdots$};
      \node(4) at (4,0) {$\bullet$};
      \node(5) at (5,0) {$\bullet$}; 
      \node(6) at (6,0) {$\bullet$}; 
      \draw (0,0) to (3);
      \draw (3) to (6,0);
      \node at (0.north) {$x_0$}; 
      \node at (1.north) {$x_1$}; 
      \node at (2.north) {$x_2$}; 
      \node at (4.north) {$x_{\ell-1}$}; 
      \node at (5.north) {$x_{\ell}$}; 
      \node at (1.south) {$y_0$}; 
      \node at (2.south) {$y_1$}; 
      \node at (4.south) {$y_{\ell'-2}$}; 
      \node at (5.south) {$y_{\ell'-1}$}; 
      \node at (6.south) {$y_{\ell'}$}; 
\end{tikzpicture}
\end{equation*}
    \item[\emph{(A3)}] We have $\ell=\ell'+2$ and $x_{i+1}=y_i$ for each $0\leq i\leq \ell'$.
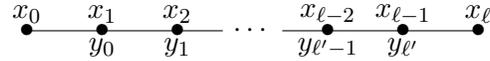
    \begin{equation*}
\begin{tikzpicture}[baseline=-\the\dimexpr\fontdimen22\textfont2\relax ]
      \node(0) at (0,0) {$\bullet$};
      \node(1) at (1,0) {$\bullet$};
      \node(2) at (2,0) {$\bullet$};
      \node(3) at (3,0) {$\cdots$};
      \node(4) at (4,0) {$\bullet$};
      \node(5) at (5,0) {$\bullet$}; 
      \node(6) at (6,0) {$\bullet$}; 
      \draw (0,0) to (3);
      \draw (3) to (6,0);
      \node at (0.north) {$x_0$}; 
      \node at (1.north) {$x_1$}; 
      \node at (2.north) {$x_2$}; 
      \node at (4.north) {$x_{\ell-2}$}; 
      \node at (5.north) {$x_{\ell-1}$}; 
      \node at (6.north) {$x_{\ell}$}; 
      \node at (1.south) {$y_0$}; 
      \node at (2.south) {$y_1$}; 
      \node at (4.south) {$y_{\ell'-1}$}; 
      \node at (5.south) {$y_{\ell'}$}; 
\end{tikzpicture}
\end{equation*}
    \item[\emph{(A4)}] We have $\ell=\ell'>0$ and $x_{i}=y_{i+1}$ for each $0\leq i\leq \ell-1$.
    
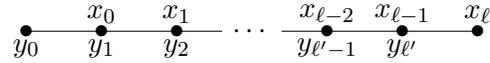
\begin{equation*}
\begin{tikzpicture}[baseline=-\the\dimexpr\fontdimen22\textfont2\relax ]
      \node(0) at (0,0) {$\bullet$};
      \node(1) at (1,0) {$\bullet$};
      \node(2) at (2,0) {$\bullet$};
      \node(3) at (3,0) {$\cdots$};
      \node(4) at (4,0) {$\bullet$};
      \node(5) at (5,0) {$\bullet$}; 
      \node(6) at (6,0) {$\bullet$}; 
      \draw (0,0) to (3);
      \draw (3) to (6,0);
      \node at (1.north) {$x_0$}; 
      \node at (2.north) {$x_1$}; 
      \node at (4.north) {$x_{\ell-2}$}; 
      \node at (5.north) {$x_{\ell-1}$}; 
      \node at (6.north) {$x_{\ell}$}; 
      \node at (0.south) {$y_0$}; 
      \node at (1.south) {$y_1$}; 
      \node at (2.south) {$y_2$}; 
      \node at (4.south) {$y_{\ell'-1}$}; 
      \node at (5.south) {$y_{\ell'}$}; 
\end{tikzpicture}
\end{equation*}
    \item[\emph{(A5)}] We have $\ell=\ell'=0$ and $x_{0}y_0\in E(H)$.
\end{itemize}
\end{proposition}
\begin{proof}
    If any of (A1)--(A5) is satisfied, then $\xi$ and $\eta$ are clearly adjacent.

    The converse can be proved by straightforward case analysis. Assume that $\xi=(x_0,\dots,x_\ell)$ and $\eta=(y_0,\dots, y_{\ell'})$ are adjacent. 
    Using $\xi=(x_0,y_0)\cdot \eta\cdot (y_{\ell'},x_\ell)$, we see the following.
    \begin{itemize}
        \item When $\ell'=0$, we have
    \begin{align*}
        \xi
        &=(x_0,y_0)\cdot (y_0)\cdot (y_0,x_\ell)\\
        &=
    \begin{cases}
        (x_0,y_0,x_\ell) &\text{if $x_0\neq x_\ell$;\quad($\implies$ (A3))}\\
        (x_0)            &\text{if $x_0= x_\ell$.\quad($\implies$ (A5))}
    \end{cases}
    \end{align*}
        \item When $\ell'=1$, we have
    \begin{align*}
        \xi
        &=(x_0,y_0)\cdot (y_0,y_1)\cdot (y_1,x_\ell)\\
        &=
        \begin{cases}
    (x_0,y_0,y_1,x_\ell) &\text{if $x_0\neq y_1$ and $y_0\neq x_\ell$;\quad($\implies$ (A3))}\\
    (x_0,x_{\ell})   &\text{if $x_0= y_1$ or $y_0= x_\ell$.\quad($\implies$ (A4) or (A2))}
    \end{cases}
    \end{align*}
        \item When $\ell'=2$, we have
    \begin{align*}
        \xi
        &=(x_0,y_0)\cdot (y_0,y_1,y_2)\cdot (y_2,x_\ell)\\
        &=
        \begin{cases}
    (x_0,y_0,y_1,y_2,x_{\ell}) &\text{if $x_0\neq y_1$ and $y_1\neq x_\ell$;\quad($\implies$ (A3))}\\
    (x_0,y_0,y_1)         &\text{if $x_0\neq y_1$ and $y_1=x_\ell$;\quad($\implies$ (A2))}\\
    (y_1,y_2,x_\ell)         &\text{if $x_0= y_1$  and  $y_1\neq x_\ell$;\quad($\implies$ (A4))}\\
    (y_1)               &\text{if $x_0=y_1=x_\ell$.\quad($\implies$ (A1))}
    \end{cases}
    \end{align*}
        \item When $\ell'\geq 3$, we have 
    \begin{align*}
        \xi
        &=(x_0,y_0)\cdot (y_0,y_1,\dots, y_{\ell'-1},y_{\ell'})\cdot (y_{\ell'},x_\ell)\\
        &=
        \begin{cases}
    (x_0,y_0,y_1,\dots,y_{\ell'-1},y_{\ell'}, x_{\ell}) &\text{if $x_0\neq y_1$ and $y_{\ell'-1}\neq x_\ell$;\quad($\implies$ (A3))}\\
    (x_0,y_0,y_1,\dots, y_{\ell'-2},y_{\ell'-1})          &\text{if $x_0\neq y_1$ and $y_{\ell'-1}=x_\ell$;\quad($\implies$ (A2))}\\
    (y_1,y_2,\dots, y_{\ell'-1},y_{\ell'},x_\ell)     &\text{if $x_0= y_1$  and  $y_{\ell'-1}\neq x_\ell$;\quad($\implies$ (A4))}\\
    (y_1,y_2,\dots, y_{\ell'-2},y_{\ell'-1})          &\text{if $x_0= y_1$ and $y_{\ell'-1}=x_\ell$.\quad($\implies$ (A1))}
    \end{cases}
    \end{align*}
    \end{itemize}
    In each case, at least one of (A1)--(A5) holds, as indicated above.
\end{proof}

\begin{corollary}\label{prop:adjacent-walk-parity}
    Let $\xi$ and $\eta$ be adjacent reduced walks in a graph $H$. Then the parities of $\len(\xi)$ and $\len(\eta)$ coincide (i.e., $\len(\xi)$ is even if and only if $\len(\eta)$ is). 
\end{corollary}

\begin{remark}\label{rmk:adjacency-between-non-reduced-walks}
    Conditions (A1)--(A5) in \cref{prop:adjacency-cases} allow one to define an adjacency relation between not necessarily reduced walks. 
    It is not hard to see that if $\xi$ and $\eta$ are walks in a graph $H$ which are adjacent (in the sense that they satisfy one of the conditions (A1)--(A5)), then so are the reduced walks $\overline \xi$ and $\overline \eta$. Also, for any graph homomorphism $f\colon G\to H$ and any adjacent pair of walks $\omega$ and $\omega'$ in $G$, the walks $f(\omega)$ and $f(\omega')$ in $H$ are clearly adjacent. It follows that $f$ induces a graph homomorphism $\Pi f\colon \Pi G\to \Pi H$ mapping each reduced walk $\omega$ in $G$ to the reduced walk $(\Pi f)(\omega)=\overline{f(\omega)}$ in $H$.
\end{remark}

The cases (A1)--(A5) in \cref{prop:adjacency-cases} are not completely mutually exclusive. That is, it is possible for reduced walks $\xi$ and $\eta$ in a graph $H$ to satisfy both (A2) and (A4). (For example, let $\xi=(x,y)$ and $\eta=(y,x)$.) However, we have the following. 
\begin{proposition}\label{prop:adjacency-walk-cases}
    Let $\xi$ and $\eta$ be adjacent reduced walks in a graph $H$. Suppose that $\len(\xi)\neq 1$ or $\len(\eta)\neq 1$. Then precisely one of \textup{(A1)--(A5)} holds.
\end{proposition}
\begin{proof}
    This follows from an inspection of the proof of \cref{prop:adjacency-cases}.
    A direct proof is as follows. We may suppose $\len(\xi)=\len(\eta)=\ell\geq 2$ and that (A2) holds; we want to show that (A4) does not hold. This follows from $x_0\neq x_2=y_1$. 
\end{proof}
\cref{prop:adjacency-walk-cases} allows us to introduce the following notion, which will play a crucial role in \cref{sec:contractibility}.
\begin{definition}\label{def:type}
    Let $\xi$ and $\eta$ be adjacent reduced walks in a graph $H$ satisfying $\len(\xi)\neq 1$ or $\len(\eta)\neq 1$. For $i\in\{1,\dots,5\}$, we say that the \emph{type} of the ordered pair $(\xi,\eta)$ is (A$i$) if $\xi$ and $\eta$ satisfy (A$i$) of \cref{prop:adjacency-cases}.
\end{definition}

Note in particular that the type is well-defined for any ordered pair of adjacent reduced walks of even lengths.

\begin{proposition}\label{prop:composition-on-PiH-is-graph-hom}
    Let $\xi_1$, $\xi_2$, $\eta_1$, and $\eta_2$ be reduced walks in a graph $H$. Suppose that $\xi_1$ and $\eta_1$ are adjacent, $\xi_2$ and $\eta_2$ are adjacent, $t(\xi_1)=s(\xi_2)$, and $t(\eta_1)=s(\eta_2)$. Then the reduced walks $\xi_1\cdot \xi_2$ and $\eta_1\cdot \eta_2$ are adjacent. 
\end{proposition}
\begin{proof}
    We have $s(\xi_1\cdot \xi_2)=s(\xi_1)$ and $s(\eta_1\cdot \eta_2)=s(\eta_1)$, and these are adjacent in $H$ because $\xi_1$ and $\eta_1$ are adjacent. Similarly, $t(\xi_1\cdot \xi_2)$ and $t(\eta_1\cdot \eta_2)$ are adjacent in $H$. Now observe 
    \begin{align*}
        \xi_1\cdot \xi_2
        &=\bigl(s(\xi_1),s(\eta_1)\bigr)\cdot \eta_1\cdot \bigl(t(\eta_1),t(\xi_1)\bigr)\cdot\bigl(s(\xi_2),s(\eta_2)\bigr)\cdot \eta_2\cdot \bigl(t(\eta_2),t(\xi_2)\bigr)\\
        &=\bigl(s(\xi_1),s(\eta_1)\bigr)\cdot \eta_1\cdot \eta_2\cdot \bigl(t(\eta_2),t(\xi_2)\bigr).\qedhere
    \end{align*}
\end{proof}

We introduce a few natural graph homomorphisms associated with $\Pi H$.

\begin{definition}\label{def:s-t-i}
    For any graph $H$, we have the following graph homomorphisms.
    \begin{itemize}
        \item Two graph homomorphisms $s,t\colon \Pi H\to H$ mapping $\xi=(x_0,\dots,x_\ell)\in V(\Pi H)$ to $s(\xi)=x_0$ and $t(\xi)=x_\ell$, respectively.
        \item A graph homomorphism $e\colon H\to \Pi H$ mapping $x\in V(H)$ to $e(x)=(x)\in V(\Pi H)$, a walk of length $0$.
    \end{itemize}
    For any graph $G$ and any graph homomorphism $f\colon G\to H$, we denote the composite graph homomorphism $ef\colon G\to H\to \Pi H$ by $\id_f\colon G\to \Pi H$.
\end{definition}

The graph $\Pi H$ may be regarded as a combinatorial analogue of the \emph{path space} $X^{[0,1]}$ of a topological space $X$. The graph homomorphisms $s,t\colon \Pi H\to H$ then correspond to the continuous maps $\mathrm{ev}_0,\mathrm{ev}_1\colon X^{[0,1]}\to X$ given by evaluation at the endpoints $0$ or $1$. 
In particular, using the graph $\Pi H$ and the graph homomorphisms $s,t\colon \Pi H\to H$, we can define the notion of homotopy between graph homomorphisms as follows. 
\begin{definition}\label{def:homotopy}
    Let $G$ and $H$ be graphs and $f,g\colon G\to H$ graph homomorphisms. A \emph{homotopy from $f$ to $g$} is a graph homomorphism $h\colon G\to \Pi H$ satisfying $sh=f$ and $th=g$.
\end{definition}
This notion will be used in \cref{subsec:relation-to-Wrochna}. Note in particular that $\id_f\colon G\to \Pi H$, which is a homotopy from $f$ to itself, can be regarded as the identity homotopy on $f$.

\begin{remark}\label{rmk:PiH-as-internal-groupoid}
    This digression is intended for readers familiar with the notion of internal category \cite[Section~XII.1]{MacLane-CWM} in a category with pullbacks. 
    For any graph $H$, 
    let $\Pi H\times_H\Pi H$ be the following pullback in the category of graphs and graph homomorphisms (in the sense of \cref{subsec:graphs,subsec:graph-hom}):
\[
\begin{tikzpicture}[baseline=-\the\dimexpr\fontdimen22\textfont2\relax ]
      \node(0) at (0,1) {$\Pi H\times_H\Pi H$};
      \node(1) at (2.5,1) {$\Pi H$};
      \node(2) at (0,-1) {$\Pi H$};
      \node(3) at (2.5,-1) {$H$.};
      \draw [->] (0) to node[auto,labelsize] {} (1);
      \draw [<-] (2) to node[auto,labelsize] {} (0);
      \draw [->] (1) to node[auto,labelsize] {$s$} (3);
      \draw [->] (2) to node[auto,labelsize] {$t$} (3);
      \draw (0.3,0.4) to (0.6,0.4) to (0.6,0.7);
\end{tikzpicture}
\]
    Explicitly, $\Pi H\times_H\Pi H$ is the induced subgraph of $\Pi H\times \Pi H$ determined by
    \[
    V(\Pi H\times_H\Pi H)=\bigl\{\,(\xi_1,\xi_2)\in V(\Pi H)\times V(\Pi H) \mid t(\xi_1)=s(\xi_2)\,\bigr\}.
    \]
    \cref{prop:composition-on-PiH-is-graph-hom} implies that we have a graph homomorphism $m\colon \Pi H\times_H\Pi H\to \Pi H$ mapping each $(\xi_1,\xi_2)\in V(\Pi H\times_H\Pi H)$ to $\xi_1\cdot \xi_2\in V(\Pi H)$. The tuple $(s,t,e,m)$ of graph homomorphisms defines an internal category (in fact, internal groupoid) structure on $\Pi H$ in the category of graphs.
\end{remark}

The two graph homomorphisms $s,t\colon \Pi H\to H$ induce a graph homomorphism $\langle s,t\rangle\colon \Pi H\to H\times H$
mapping each $\xi\in V(\Pi H)$ to $\bigl(s(\xi),t(\xi)\bigr)\in V(H\times H)$, by the universal property of the (categorical) product. 

\begin{proposition}
\label{prop:st-covering}
    Let $H$ be a graph. 
    Then the graph homomorphism $\langle s,t\rangle\colon \Pi H\to H\times H$ is a covering map of graphs.
\end{proposition}
\begin{proof}
    Given any reduced walk $\xi\in V(\Pi H)$, the function 
    \[
    \langle s,t\rangle|_{N_{\Pi H}(\xi)}\colon N_{\Pi H}(\xi)\to N_{H\times H}\bigl(\bigl(s(\xi),t(\xi)\bigr)\bigr)=N_H\bigl(s(\xi)\bigr)\times N_H\bigl(t(\xi)\bigr)
    \]
    is a bijection, its inverse 
    \[
    \bigl(\langle s,t\rangle|_{N_{\Pi H}(\xi)}\bigr)^{-1}\colon N_H\bigl(s(\xi)\bigr)\times N_H\bigl(t(\xi)\bigr)\to  N_{\Pi H}(\xi)
    \]
    being given by mapping $(x,y)\in N_H\bigl(s(\xi)\bigr)\times N_H\bigl(t(\xi)\bigr)$ to 
    \[
    \bigl(\langle s,t\rangle|_{N_{\Pi H}(\xi)}\bigr)^{-1}\bigl((x,y)\bigr)= \bigl(x,s(\xi)\bigr)\cdot \xi\cdot \bigl(t(\xi),y\bigr).\qedhere
    \]
\end{proof}

\subsection{The definition of the universal cover}\label{subsec:def-of-pf}
Let $G$ and $H$ be graphs and let $f\colon G\to H$ be a graph homomorphism.
In this subsection,  we construct a certain monotone map $p_f\colon E_f\to B_f$, where $B_f$ is the connected component $\Conn\bigl(\Hom(G,H),f\bigr)$ of $\Hom(G,H)$ containing $f$. 
When the graphs $G$ and $H$ satisfy the assumption of \cref{thm:main}, $p_f$ will be the universal cover of $B_f$. 

Our definition of $E_f$ uses the graph $\Pi H$. 
The graph homomorphism $t\colon \Pi H\to H$ induces a monotone map $\Hom(G,t)\colon \Hom(G,\Pi H)\to \Hom(G,H)$ mapping each $\varphi\in \Hom(G,\Pi H)$ to $t\varphi\in \Hom(G,H)$. 
We denote by $\Hom(G,\Pi H)_f$ the subposet of $\Hom(G,\Pi H)$ consisting of all set-valued homomorphisms $\varphi\colon G\pto\Pi H$ such that $s\varphi=f$.
Recall from \cref{def:s-t-i} that we have a graph homomorphism $\id_f\colon G\to \Pi H$ mapping each $u\in V(G)$ to the walk $\bigl(f(u)\bigr)\in V(\Pi H)$ of length $0$. We clearly have $\id_f\in \Hom(G,\Pi H)_f$. Since the monotone map $\Hom(G,t)\colon \Hom(G,\Pi H)\to \Hom(G,H)$ maps $\id_f$ to $f$, it restricts to 
\[
\Hom(G,t)|_{\Conn(\Hom(G,\Pi H)_f,\id_f)}\colon \Conn\bigl(\Hom(G,\Pi H)_f,\id_f\bigr)\to \Conn\bigl(\Hom(G,H),f\bigr)=B_f.
\]
We define $E_f=\Conn\bigl(\Hom(G,\Pi H)_f,\id_f\bigr)$ and $p_f=\Hom(G,t)|_{E_f}\colon E_f\to B_f$.
See \cref{apx:comparing} for another definition of $E_f$ based on universal covers of graphs. 

As mentioned above, if $G$ and $H$ satisfy the assumption of \cref{thm:main}, then $p_f\colon E_f\to B_f$ is the universal cover of $B_f$ by \cref{thm:cover,thm:contractible}.
We will give a more explicit description of $E_f$ in \cref{prop:univ-cover-explicitly}.
In the proof of \cref{thm:Lambda-classification}, we show that $E_f$ is closely related to the notion of \emph{realizable walks} in \cite{Wro20}.

\section{The map $p_f$ is a covering map} \label{sec:covering}
The aim of this section is to show the following.

\begin{theorem}\label{thm:cover}
    Let $G$ and $H$ be graphs satisfying the assumption of \cref{thm:main} and $f\colon G\to H$ a graph homomorphism. Then the map $p_f\colon E_f\to B_f$ defined in \cref{subsec:def-of-pf} is a covering map.
\end{theorem}

The above statement follows our convention (see \cref{subsec:poset-topology}) to identify a poset $P$ and its classifying space $\gr{P}$. So it really says that the continuous map $\gr{p_f}\colon \gr{E_f}\to \gr{B_f}$ is a covering map between topological spaces (see e.g.~\cite[Section~1.3]{Hat02} for the definition). However, as we will see in \cref{subsec:cover-poset-topology}, it can also be regarded as a statement about the monotone map $p_f\colon E_f\to B_f$ between posets; and it is this latter combinatorial statement that we actually prove.

\subsection{Results from poset topology}\label{subsec:cover-poset-topology}
We begin with recalling relevant notions and a result from poset topology.
\begin{definition}[Cf.~{\cite[Definition~2.5]{Barmak-Minian-covering}}]
    Let $P$ and $Q$ be posets and $\alpha\colon P\to Q$ a monotone map. 
    \begin{itemize}
        \item $\alpha$ is a \emph{discrete fibration} if, for each $p\in P$ and each $q\in Q$ with $q\leq \alpha (p)$, there exists a unique $p'\in P$ with $p'\leq p$ and $\alpha(p')=q$.
        \item $\alpha$ is a \emph{discrete opfibration} if, for each $p\in P$ and each $q\in Q$ with $\alpha(p)\leq q$, there exists a unique $p'\in P$ with $p\leq p'$ and $\alpha(p')=q$.
        \item $\alpha$ is a \emph{covering map of posets} if it is both a discrete fibration and a discrete opfibration.\qedhere
    \end{itemize}
\end{definition} 

\begin{proposition}[{\cite[Proposition~2.9 and Theorem~3.2]{Barmak-Minian-covering}}]
    Let $P$ and $Q$ be posets and $\alpha\colon P\to Q$ a covering map of posets. Then $\gr{\alpha}\colon \gr{P}\to \gr{Q}$ is a covering map between topological spaces.
\end{proposition}

Notice that for any covering map $\alpha\colon P\to Q$ of posets and any $p\in P$, $\alpha|_{\Conn(P,p)}\colon \Conn(P,p)\to \Conn\bigl(Q,\alpha(p)\bigr)$ is also a covering map of posets.

\subsection{The proof of \cref{thm:cover}}
By \cref{subsec:cover-poset-topology}, it suffices to show that the monotone map $\Hom(G,t)|_{\Hom(G,\Pi H)_f}\colon \Hom(G,\Pi H)_f\to \Hom(G,H)$ is a covering map of posets.

We begin with the following partial result.
\begin{proposition}
\label{prop:HomGt-discrete-fibration}
    Let $G$ be a graph without isolated vertices, $H$ a graph, and $f\colon G\to H$ a graph homomorphism.
    Then the monotone map 
    \[
    \Hom(G,t)|_{\Hom(G,\Pi H)_f}\colon \Hom(G,\Pi H)_f\to \Hom (G,H)
    \]
    is a discrete fibration and satisfies the ``uniqueness part'' of the definition of discrete opfibration: for any $\varphi\in \Hom(G,\Pi H)_f$ and any $\psi\in \Hom(G,H)$ with $t\varphi\leq \psi$, there exists at most one $\varphi'\in \Hom(G,\Pi H)_f$ satisfying $\varphi\leq \varphi'$ and $t\varphi'=\psi$. 
\end{proposition}
\begin{proof}
    First we observe that for each $\varphi\in\Hom(G,\Pi H)_f$ and each $u\in V(G)$, the function $t|_{\varphi(u)}\colon \varphi(u)\to t\varphi(u)$ is bijective. It is surjective since $t\varphi(u)$ is defined to be the image $t\bigl(\varphi(u)\bigr)\subseteq V(H)$.
    To see its injectivity, let $v\in V(G)$ be a vertex adjacent to $u$ and take any $\eta\in \varphi(v)$. Then any $\xi\in \varphi(u)$ satisfies $s(\xi)=f(u)$ (by the definition of $\Hom(G,\Pi H)_f$) and is adjacent to $\eta$. Therefore \cref{prop:st-covering} implies that $t|_{\varphi(u)}\colon \varphi(u)\to t\varphi(u)$ is injective. Indeed, the inverse $\bigl(t|_{\varphi(u)}\bigr)^{-1}\colon t\varphi(u)\to \varphi(u)$ is obtained by mapping $x\in t\varphi(u)$ to 
    \[
    \bigl(t|_{\varphi(u)}\bigr)^{-1}(x)=\bigl(f(u),f(v)\bigr) \cdot \eta \cdot \bigl(t(\eta),x\bigr).
    \]

    Now let $\varphi \in \Hom(G,\Pi H)_f$ and take any $\psi\in \Hom (G,H)$ with $\psi\leq t\varphi$. 
    Define the function $\varphi'\colon V(G)\to\Pow\bigl(V(\Pi H)\bigr)\setminus\{\emptyset\}$ by $\varphi'(u)=\bigl(t|_{\varphi(u)}\bigr)^{-1}\bigl(\psi(u)\bigr)$ for each $u\in V(G)$; in other words, we set 
\[
\varphi'(u)=\bigl\{\,\bigl(f(u),f(v)\bigr) \cdot \eta \cdot \bigl(t(\eta),x\bigr)\;\big\vert\;x\in \psi(u)\,\bigr\}
\]
    for each $u\in V(G)$, where $v\in V(G)$ is a vertex adjacent to $u$ and $\eta\in \varphi(v)$.
    Then $\varphi'$ is a set-valued homomorphism $G\pto \Pi H$ because we have $\varphi'\leq \varphi$ and $\varphi\colon G\pto \Pi H$. Since $t|_{\varphi(u)}$ is injective for each $u\in V(G)$, $\varphi'$ is clearly the unique element of $\Hom(G,\Pi H)_f$ satisfying $\varphi'\leq \varphi$ and $t\varphi'=\psi$.

    Finally, to show the second statement, take any $\varphi\in \Hom(G,\Pi H)_f$ and any $\psi\in \Hom(G,H)$ with $t\varphi\leq \psi$. Suppose that there exists $\varphi'\in \Hom(G,\Pi H)_f$ satisfying $\varphi\leq \varphi'$ and $t\varphi'=\psi$.  Take any $u\in V(G)$, and choose any vertex $v\in V(G)$ adjacent to $u$ and any $\eta\in \varphi(v)$. Then each element of $\varphi'(u)$ must be adjacent to $\eta$, and hence $\varphi'(u)$ is forced to be 
    \[
    \varphi'(u)=\bigl\{\,\bigl(f(u),f(v)\bigr)\cdot \eta\cdot \bigl(t(\eta),x\bigr)\;\big\vert\; x\in \psi(u)\,\bigr\}.
    \]
    This shows the uniqueness of $\varphi'$.
\end{proof}

\begin{remark}
The monotone map 
\[\Hom(G,t)|_{\Hom(G,\Pi H)_f}\colon \Hom(G,\Pi H)_f\to \Hom (G,H)\] 
may fail to be a covering map of posets when $H$ is not square-free.
For example, let $G$ and $H$ be the following graphs. 
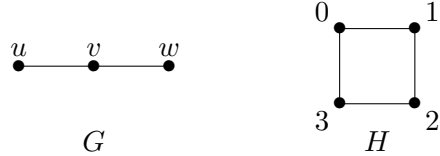
\[
\begin{tikzpicture}[baseline=-\the\dimexpr\fontdimen22\textfont2\relax ]
      \node(0) at (0,0) {$\bullet$};
      \node(1) at (1,0) {$\bullet$};
      \node(2) at (2,0) {$\bullet$};
      \node at (1,-1) {$G$};
      \draw (0,0) to (2,0);
      \node at (0.north) {$u$}; 
      \node at (1.north) {$v$}; 
      \node at (2.north) {$w$}; 
\end{tikzpicture}
\qquad\qquad
\begin{tikzpicture}[baseline=-\the\dimexpr\fontdimen22\textfont2\relax ]
      \node(0) at (0,0.5) {$\bullet$};
      \node(1) at (1,0.5) {$\bullet$};
      \node(2) at (1,-0.5) {$\bullet$};
      \node(3) at (0,-0.5) {$\bullet$}; 
      \draw (0,0.5) to (1,0.5) to (1,-0.5) to (0,-0.5) to (0,0.5);
      \node at (0.north west) {$0$}; 
      \node at (1.north east) {$1$}; 
      \node at (2.south east) {$2$}; 
      \node at (3.south west) {$3$}; 
      \node at (0.5,-1) {$H$};
\end{tikzpicture}
\]
Define a graph homomorphism $f\colon G\to H$ by $f(u)=0$, $f(v)=1$, and $f(w)=2$. 
Consider the graph homomorphism $\id_f\colon G\to \Pi H$ (see \cref{def:s-t-i}); we have $\id_f\in \Hom(G,\Pi H)_f$.
Now define a set-valued homomorphism $\psi\colon G\pto H$ by 
$\psi(u)=\{0\}$, $\psi(v)=\{1,3\}$, and $\psi(w)=\{2\}$.
This $\psi$ satisfies $t\;\id_f =f\leq \psi$, but there is no $\varphi\in \Hom(G,\Pi H)_f$ with $\id_f\leq \varphi$ and $t\varphi=\psi$. Indeed, we cannot put the reduced walk $(1,2,3)$ into $\varphi(v)$ since it is not adjacent to $(0)=\id_f(u)$, and similarly we cannot put $(1,0,3)$ into $\varphi(v)$ since it is not adjacent to $(2)=\id_f(w)$.
\end{remark}

The following fact is related to the so-called \emph{monochromatic neighborhood property} \cite{Wro20,LMS25} in combinatorial reconfiguration.

\begin{lemma}
\label{lem:sq-free-singleton}
    Let $G$ and $H$ be graphs satisfying the assumption of \cref{thm:main} and $\psi\colon G\pto H$ a set-valued homomorphism. Suppose that $u\in V(G)$ is a vertex such that $\psi(u)$ has at least two elements.
    Then there exists a (unique) vertex $z\in V(H)$ such that $\psi(v)=\{z\}$ for each $v\in N_G(u)$.
\end{lemma}
\begin{proof}
    Suppose that $\psi(u)$ contains $x$ and $y$ with $x\neq y$. Then we have 
    \[
    \emptyset\subsetneq\psi(v)\subseteq N_H(x)\cap N_H(y)
    \]
    for each $v\in N_G(u)$. Since $N_G(u)\neq \emptyset$, by \cref{eqn:sq-free} there exists $z\in V(H)$ such that $N_H(x)\cap N_H(y)=\{z\}$.
\end{proof}

\cref{thm:cover} is a consequence of the following.
\begin{proposition}
\label{thm:discrete-bifibration}
    Let $G$ and $H$ be graphs satisfying the assumption of \cref{thm:main} and $f\colon G\to H$ a graph homomorphism. Then the monotone map 
    \[
    \Hom(G,t)|_{\Hom(G,\Pi H)_f}\colon \Hom(G,\Pi H)_f\to \Hom (G,H)
    \]
    is a covering map of posets.
\end{proposition}
\begin{proof}
    Let $\varphi \in \Hom(G,\Pi H)_f$ and take any $\psi\in \Hom (G,H)$ with $t\varphi\leq \psi$. 
    In view of \cref{prop:HomGt-discrete-fibration}, it suffices to show that there exists \emph{some} $\varphi' \in \Hom(G,\Pi H)_f$ with $\varphi\leq \varphi'$ and $t\varphi'=\psi$.

    This clearly holds when $t\varphi = \psi$. So suppose that there exists $u\in V(G)$ with $t\bigl(\varphi(u)\bigr)\subsetneq\psi(u)$.
    For each such $u$, by \cref{lem:sq-free-singleton} there exists $z_u\in V(H)$ such that $\psi(v)=\{z_u\}$ holds for each $v\in N_G(u)$.
    Choose and fix $\eta_u\in \varphi(u)$ arbitrarily and let $y_u=t(\eta_u)$. Now define 
    \[
    \varphi'(u)=\begin{cases}
        \varphi(u) &\text{if $t\bigl(\varphi(u)\bigr)=\psi(u)$, and}\\
        \bigl\{\,\eta_u \cdot (y_u, z_u,x)\;\big\vert\; x\in \psi(u) \,\bigr\} & \text{if $t\bigl(\varphi(u)\bigr)\subsetneq\psi(u)$}.
    \end{cases}
    \]
    To see that $\varphi'$ is a set-valued homomorphism $G\pto \Pi H$, take any $u\in V(G)$ such that $t\bigl(\varphi(u)\bigr)\subsetneq\psi(u)$. Then for any $v\in N_G(u)$, $\varphi(v)$ is a singleton whose unique element is 
    \[
    \bigl(f(v),f(u)\bigr)\cdot \eta_u\cdot (y_u,z_u).
    \]
    This reduced walk is equal to 
    \[
    \bigl(f(v),f(u)\bigr)\cdot \eta_u \cdot (y_u,z_u,x) \cdot (x,z_u)
    \]
    for each $x\in \psi(u)$, and hence is adjacent to any element in $\varphi'(u)$. 
    This shows that $\varphi'$ is a set-valued homomorphism. We clearly have $\varphi'\in\Hom(G,\Pi H)_f$, $\varphi\leq\varphi'$, and $t\varphi'=\psi$. 
\end{proof}

\section{Contractibility of $E_f$}\label{sec:contractibility}
Throughout this section, let $G$ and $H$ be graphs satisfying the assumption of \cref{thm:main}, and let $f\colon G\to H$ be a graph homomorphism.
The aim of this section is to show the following.
\begin{theorem}\label{thm:contractible}
    Let $G$ and $H$ be graphs satisfying the assumption of \cref{thm:main} and $f\colon G\to H$ a graph homomorphism. 
    Then $E_f$ defined in \cref{subsec:def-of-pf} is contractible.
\end{theorem}
In \cref{thm:contractible}, we are claiming the contractibility of the topological space $\gr{E_f}$ associated with the poset $E_f$ (cf.~\cref{subsec:poset-topology}). However, we will use results of poset topology (recalled in \cref{subsec:poset-topology-contractible}) and deduce the above claim via combinatorial arguments on the poset $E_f$.
We note that the proof of \cref{thm:contractible} given in \cref{subsec:proof-of-contractibility} is based on Matsushita’s idea \cite{Mat25}.

\subsection{Explicit description of $E_f$}
We first give a more explicit description of $E_f$ as a subposet of $\Hom(G,\Pi H)_f$. To this end, we use the following notions.

\begin{definition}[{\cite[Section~5]{Wro20}}]\label{def:f-tight}
    A closed walk $\xi=(x_0,x_1,\dots,x_k=x_0)$ in the graph $H$ is \emph{cyclically reduced} if it is reduced, $k\geq 3$, and $x_{k-1}\neq x_{1}$ holds. 
    When dealing with such closed walks, we often adopt the following notational convention to ease the handling of indices. 
    Let $\mathbb{Z}/k\mathbb{Z}$ denote the set $\{\,0,\dots,k-1\,\}$, and we perform the operation of ``addition modulo $k$'' freely on elements $i\in\mathbb{Z}/k\mathbb{Z}$ (so that for example, when $i=k-1$, we have $i+2=1$). 
    Using this, we can express the condition for a closed walk $\xi=(x_0,x_1,\dots,x_k=x_0)$ in $H$ to be cyclically reduced as: $k\geq 3$ and we have $x_{i}\neq x_{i+2}$ for each $i\in \mathbb{Z}/k\mathbb{Z}$. 
    
    A closed walk $\omega$ in the graph $G$ is \emph{$f$-tight} if the closed walk $f(\omega)$ in $H$ is cyclically reduced.
\end{definition}

\begin{definition}\label{def:norm-h}
    For any graph homomorphism $h\colon G\to \Pi H$, define $\norm{h}\in\mathbb{N}$ as 
    \[
    \norm{h}=\sum_{u\in V(G)}\len\bigl(h(u)\bigr).\qedhere
    \]
\end{definition}

Using these notions, we can characterize $E_f$ as follows.

\begin{proposition}\label{prop:univ-cover-explicitly}
    A set-valued homomorphism $\varphi\colon G\pto \Pi H$ with
    $\varphi\in \Hom(G,\Pi H)_f$ is in $E_f$ if and only if the following conditions are satisfied.
    \begin{enumerate}[label=\emph{({\arabic*})}]
        \item For each $u\in V(G)$ and each $\xi\in\varphi(u)$, $\len(\xi)$ is even.
        \item For each $u\in V(G)$ which is in an $f$-tight closed walk, we have $\varphi(u)=\bigl\{\bigl(f(u)\bigr)\bigr\}$.
    \end{enumerate}
\end{proposition}

\cref{prop:univ-cover-explicitly} is a direct consequence of the following. 

\begin{proposition}[{Cf.~\cite[Theorem~6.1]{Wro20} and \cite[Proposition~4.5]{cycle}}]\label{prop:univ-cover-explicitly-atom}
    A graph homomorphism $h\colon G\to \Pi H$ with $h\in\Hom(G,\Pi H)_f$ is in $E_f$ if and only if the following conditions are satisfied.
    \begin{enumerate}[label=\emph{({\arabic*})}]
        \item For each $u\in V(G)$, $\len\bigl(h(u)\bigr)$ is even.
        \item For each $u\in V(G)$ which is in an $f$-tight closed walk, we have $h(u)=\bigl(f(u)\bigr)$.
    \end{enumerate}
\end{proposition}

The proof of (the ``if'' part of) \cref{prop:univ-cover-explicitly-atom} is based on the following construction. 

\begin{definition}[{Cf.~\cite[Proof of Theorem~6.1]{Wro20}}]\label{def:Gh}
    For any graph homomorphism $h\colon G\to \Pi H$, we define the digraph $\overrightarrow{G}_h$ as follows. The set $V(\overrightarrow{G}_h)$ of vertices of $\overrightarrow{G}_h$ is given by 
    \[
    V(\overrightarrow{G}_h)=\bigl\{\,u\in V(G)\;\big\vert\; \len\bigl(h(u)\bigr)\geq 2\,\bigr\}
    \]
    and the set $A(\overrightarrow{G}_h)$ of arcs of $\overrightarrow{G}_h$ is given by 
    \begin{multline*}
    A(\overrightarrow{G}_h)=\bigl\{\,(u,v)\in V(\overrightarrow{G}_h)\times V(\overrightarrow{G}_h) \;\big\vert\; \text{$uv\in E(G)$ and}\\
    \text{the type of $\bigl(h(u),h(v)\bigr)$ is either (A1) or (A2)}\,\bigr\},
    \end{multline*}
    where the notion of type was defined in \cref{def:type}.
    Notice that in the definition of $A(\overrightarrow{G}_h)$, if $(u,v)\in V(\overrightarrow{G}_h)\times V(\overrightarrow{G}_h)$ satisfies $uv\in E(G)$, then $h(u)$ and $h(v)$ are adjacent in $\Pi H$. 
    Since the lengths of $h(u)$ and $h(v)$ are greater than $1$, the type of $\bigl(h(u),h(v)\bigr)$ is one of (A1)--(A4).  
    According to the above definition, we have $(u,v)\in A(\overrightarrow{G}_h)$ when the type of $\bigl(h(u),h(v)\bigr)$ is either (A1) or (A2).  
    Otherwise, i.e., if the type of $\bigl(h(u),h(v)\bigr)$ is either (A3) or (A4), then we have $(v,u)\in A(\overrightarrow{G}_h)$ instead. Therefore the digraph $\overrightarrow{G}_h$ is obtained from the induced subgraph of $G$ on the vertex set $V(\overrightarrow{G}_h)$ by choosing a direction of each edge. 
    Explicitly, for $u,v\in V(\overrightarrow{G}_h)$ with $uv\in E(G)$, if we let $h(u)=(x_0,\dots, x_\ell)$ and $h(v)=(y_0,\dots, y_{\ell'})$, then we have $(u,v)\in A(\overrightarrow{G}_h)$ precisely when $x_\ell=y_{\ell'-1}$ holds. 
\end{definition}

\begin{lemma}\label{lem:Gh-len}
    Let $h\colon G\to \Pi H$ be a graph homomorphism. For any $(u,v)\in A(\overrightarrow{G}_h)$, we have $\len\bigl(h(u)\bigr)\leq \len \bigl(h(v)\bigr)$.
\end{lemma}
\begin{proof}
    As explained in \cref{def:Gh}, $(u,v)\in A(\overrightarrow{G}_h)$ means that the type of $\bigl(h(u),h(v)\bigr)$ is either (A1) or (A2). Therefore we have either $\len\bigl(h(u)\bigr)+2=\len \bigl(h(v)\bigr)$ or $\len\bigl(h(u)\bigr)=\len \bigl(h(v)\bigr)$.
\end{proof}

A vertex $v$ in a digraph $\overrightarrow{A}$ is called a \emph{sink} in $\overrightarrow{A}$ if there exists no vertex $w$ in $\overrightarrow{A}$ such that $(v,w)$ is an arc in $\overrightarrow{A}$.

\begin{lemma}\label{lem:sink-of-Gh-can-be-recolored}
    Let $h\colon G\to \Pi H$ be a graph homomorphism. 
    If $v$ is a sink in $\overrightarrow{G}_h$, with $h(v)=(x_0,\dots, x_\ell)$, then the function $h'\colon V(G)\to V(\Pi H)$, defined as 
    \[
    h'(u)=\begin{cases}
        (x_0,\dots, x_{\ell-2}) &\text{if $u=v$, and}\\
        h(u)&\text{if $u\neq v$}
    \end{cases}
    \]
    for each $u\in V(G)$, is a graph homomorphism $G\to \Pi H$.
\end{lemma}
\begin{proof}
    Let $w\in V(G)$ be a vertex adjacent to $v$. Then $h(v)$ and $h(w)$ are adjacent, and the type of $\bigl(h(v),h(w)\bigr)$ is either (A3) or (A4). It follows that $h'(v)$ and $h'(w)=h(w)$ are also adjacent.
\end{proof}

\begin{lemma}\label{lem:directed-cycle-is-f-tight}
    Let $h\colon G\to \Pi H$ be a graph homomorphism with $h\in \Hom(G,\Pi H)_f$.
    Then every directed cycle in $\overrightarrow{G}_h$ is an $f$-tight closed walk.
\end{lemma}
\begin{proof}
    Let $\omega=(v_0,v_1,\dots, v_k=v_0)$ ($k\geq 3$) be a directed cycle in $\overrightarrow{G}_h$. First observe that we have $\len\bigl(h(v_0)\bigr)\leq \dots\leq \len\bigl(h(v_k)\bigr)=\len\bigl(h(v_0)\bigr)$ by \cref{lem:Gh-len}, and hence $\len\bigl(h(v_0)\bigr)= \dots = \len\bigl(h(v_k)\bigr)$, which we write as $\ell$; note that $\ell\geq 2$ holds by the definition of $V(\overrightarrow{G}_h)$. 
    In what follows, we adopt the notational convention explained in \cref{def:f-tight}. 
    Letting $h(v_i)=(x_{i,0},x_{i,1},\dots,x_{i,\ell})$ for each $i\in \mathbb{Z}/k\mathbb{Z}$, we have $f(v_i)=x_{i,0}$ for each $i\in \mathbb{Z}/k\mathbb{Z}$ by $sh=f$, and $x_{i,j+1}=x_{i+1,j}$ for each $i\in \mathbb{Z}/k\mathbb{Z}$ and each $0\leq j\leq \ell-1$ since the type of $\bigl(h(v_i),h(v_{i+1})\bigr)$ is (A2). These imply that we have $f(v_i)=x_{i,0}\neq x_{i,2}=x_{i+1,1}=x_{i+2,0}=f(v_{i+2})$ for each $i\in \mathbb{Z}/k\mathbb{Z}$ since $h(v_i)$ is a reduced walk and $\ell\geq 2$. Therefore $\omega$ is an $f$-tight closed walk. 
\end{proof}

\begin{proof}[Proof of \cref{prop:univ-cover-explicitly-atom}]
    Modulo suitable translation (cf.~\cref{prop:homotopy-top-valid-walk}) and modification, this would follow from \cite[Proof of Theorem~6.1]{Wro20};
    below we give a self-contained proof using the terminology of this paper. 

    We start with the proof of the ``only if'' part. Since $\id_f\colon G\to \Pi H$ clearly satisfies (1) and (2) of \cref{prop:univ-cover-explicitly-atom}, it suffices to show the following: given graph homomorphisms $h_1,h_2\colon G\to \Pi H$ such that $h_1,h_2\in \Hom(G,\Pi H)_f$ and that $h_1$ and $h_2$ are adjacent (see \cref{subsec:graph-hom}), if $h_1$ satisfies (1) and (2), then so does $h_2$. 
    Let $v\in V(G)$ be the unique vertex with $h_1(v)\neq h_2(v)$. 
    It suffices to check (1) and (2) for $h_2$ in the case $u=v$. As for (1), let $w\in V(G)$ be a vertex adjacent to $v$ (such a vertex exists since $G$ is connected and has at least two vertices). The reduced walk $h_1(w)=h_2(w)$ has even length since $h_1$ satisfies (1), and is adjacent to $h_2(v)$. By \cref{prop:adjacent-walk-parity}, we see that $\len\bigl(h_2(v)\bigr)$ is even. 
    To show (2), it suffices to show that $v$ cannot be in any $f$-tight closed walk in $G$. Thus suppose to the contrary that there exists an $f$-tight closed walk $(v_0,v_1,\dots,v_k=v_0)$ in $G$ with $v=v_0$. We have $h_1(v_i)=\bigl(f(v_i)\bigr)$ for each $i\in \mathbb{Z}/k\mathbb{Z}$ since $h_1$ satisfies (2). Now, by $sh_2=f$ and $\bigl(f(v)\bigr)=h_1(v)\neq h_2(v)$ we see that $h_2(v)$ must be a reduced walk of length $2$, but this cannot be adjacent to both $h_2(v_1)=\bigl(f(v_1)\bigr)$ and $h_2(v_{k-1})=\bigl(f(v_{k-1})\bigr)$ because we have $f(v_1)\neq f(v_{k-1})$ by the $f$-tightness, a contradiction. 
    
    The ``if'' part is proved by induction on $\norm{h}$. The base case is trivial: if $h\colon G\to \Pi H$ with $sh=f$ satisfies $\norm{h}=0$, then $h=\id_f$. Thus take a graph homomorphism $h\colon G\to \Pi H$ satisfying $sh=f$, (1), (2), and $\norm{h}>0$. 
    We first show that the digraph $\overrightarrow{G}_h$ has a sink. Indeed, the (finite) set  $V(\overrightarrow{G}_h)$ is nonempty by (1) and $\norm{h}>0$. 
    Hence, if $\overrightarrow{G}_h$ had no sink, then we would be able to find a directed cycle in $\overrightarrow{G}_h$, which is an $f$-tight closed walk by \cref{lem:directed-cycle-is-f-tight}. However, this contradicts the assumption that $h$ satisfies (2). 
    Thus let $v\in V(\overrightarrow{G}_h)\subseteq V(G)$ be a sink in $\overrightarrow{G}_h$ and let $h(v)=(x_0,\dots,x_\ell)$. Then the function $h'\colon V(G)\to V(\Pi H)$, defined as 
    \[
    h'(u)=\begin{cases}
        (x_0,\dots, x_{\ell-2}) &\text{if $u=v$, and}\\
        h(u) &\text{if $u\neq v$,}
    \end{cases}
    \]
    is a graph homomorphism $G\to \Pi H$ by \cref{lem:sink-of-Gh-can-be-recolored}. It is clear that $h'$ satisfies $sh'=f$, (1), and (2), and we have $\norm{h'}=\norm{h}-2$. Hence we have $h'\in E_f$ by the induction hypothesis. Since $h'$ and $h$ are adjacent, we conclude $h\in E_f$. 
\end{proof}

\begin{proof}[Proof of \cref{prop:univ-cover-explicitly}]
    This follows from \cref{prop:univ-cover-explicitly-atom} by the following observations. Let $\varphi\in \Hom(G,\Pi H)_f$. Then we have $\varphi\in E_f$ if and only if for any graph homomorphism $h\colon G\to \Pi H$ with $h\leq \varphi$, we have $h\in E_f$.
    Similarly, $\varphi$ satisfies (1) and (2) of \cref{prop:univ-cover-explicitly} if and only if any graph homomorphism $h\colon G\to \Pi H$ with $h\leq \varphi$ satisfies (1) and (2) of \cref{prop:univ-cover-explicitly-atom}.
\end{proof}

\subsection{Results from poset topology}\label{subsec:poset-topology-contractible}
Here we recall results in poset topology used in the proof of \cref{thm:contractible}. 
\begin{definition}
    Let $P$ be a poset and $\alpha\colon P\to P$ be a monotone map. 
    \begin{itemize}
        \item $\alpha$ is a \emph{closure operator} if for each $p\in P$, $p\leq \alpha(p)$ and $\alpha^2(p)=\alpha(p)$ hold.
        \item $\alpha$ is an \emph{interior operator} if for each $p\in P$, $p\geq \alpha(p)$ and $\alpha^2(p)=\alpha(p)$ hold.\qedhere
    \end{itemize}
\end{definition}

\begin{proposition}[{See e.g.~\cite[Corollary~10.12 and the subsequent sentence]{Bjorner-topology}}]\label{prop:closure-sdr}
    Let $P$ be a poset and $\alpha\colon P\to P$ be either a closure operator or an interior operator. Let $\alpha(P)$ be the image of $P$, regarded as a subposet of $P$. Then $\alpha(P)$ is a strong deformation retract of $P$.
\end{proposition}

\subsection{The proof of \cref{thm:contractible}}\label{subsec:proof-of-contractibility}
We first extend \cref{def:norm-h,def:Gh} to set-valued homomorphisms $\varphi\colon G\pto \Pi H$ in $E_f$.
\begin{definition}[{Cf.~\cref{def:norm-h}}]\label{def:norm-phi}
    For any set-valued homomorphism $\varphi\colon G\pto \Pi H$ and a vertex $u\in V(G)$, define 
    \[
    \len(\varphi,u)=\max\{\,\len(\xi)\mid\xi\in\varphi(u)\,\}
    \]
    and 
    \[
    \norm{\varphi}=\sum_{u\in V(G)}\len(\varphi,u).\qedhere
    \]
\end{definition}

\begin{lemma}\label{lem:type-well-defined-1}
    Let $\xi,\xi',\eta\in V(\Pi H)$. 
    Suppose that $\xi\eta,\xi'\eta\in E(\Pi H)$, $\len(\xi)=\len(\xi')\geq 2$, and $s(\xi)=s(\xi')$.
    Then the types of $(\xi,\eta)$ and $(\xi',\eta)$ coincide.
\end{lemma}
\begin{proof}
    It suffices to consider the case where $\len(\xi)=\len(\xi')=\len(\eta)$. Let $\ell$ be the common length of $\xi,\xi'$, and $\eta$, and let $\xi=(x_0,\dots,x_\ell)$, $\xi'=(x'_0,\dots, x'_\ell)$, and $\eta=(y_0,\dots, y_\ell)$. 
    We have $x_0=x'_0$. 
    First suppose that the type of $(\xi,\eta)$ is (A2). Then we have $x_0\neq x_2=y_1$. Therefore we have $x'_0\neq y_1$ and the type of $(\xi',\eta)$ cannot be (A4). Hence the type of  $(\xi',\eta)$ must be (A2). 
    Next suppose that the type of $(\xi,\eta)$ is (A4). Then we have $x_0=y_1$. Therefore we have $x'_2\neq x'_0=y_1$ and the type of $(\xi',\eta)$ cannot be (A2). Hence the type of  $(\xi',\eta)$ must be (A4). 
\end{proof}

\begin{corollary}\label{cor:type-well-defined}
    Let $\varphi\colon G\pto \Pi H$ be a set-valued homomorphism with $\varphi\in \Hom (G,\Pi H)_f$, and
    $u,v\in V(G)$ a pair of adjacent vertices in $G$ such that $\len(\varphi,u)\geq 2$. Then, for any $\xi,\xi'\in \varphi(u)$ with $\len(\xi)=\len(\xi')=\len(\varphi,u)$ and any $\eta\in \varphi(v)$, the types of $(\xi,\eta)$ and $(\xi',\eta)$ coincide.
\end{corollary}

\begin{corollary}\label{cor:Gh1=Gh2} 
    Let $\varphi\colon G\pto \Pi H$ be a set-valued homomorphism with $\varphi\in \Hom (G,\Pi H)_f$.
    Then for any pair of graph homomorphisms $h_1,h_2\colon G\to \Pi H$ such that $h_1\leq \varphi$, $h_2\leq \varphi$, and $\norm{h_1}=\norm{h_2}=\norm{\varphi}$, we have $\overrightarrow{G}_{h_1}=\overrightarrow{G}_{h_2}$.
\end{corollary}
\begin{proof}
    For each $u\in V(G)$ and $i\in \{1,2\}$, $h_i(u)$ is an element of $\varphi(u)$ with $\len\bigl(h_i(u)\bigr)=\len(\varphi,u)$. Therefore we have $V(\overrightarrow{G}_{h_1})=V(\overrightarrow{G}_{h_2})$.
    We also have $A(\overrightarrow{G}_{h_1})=A(\overrightarrow{G}_{h_2})$ by \cref{cor:type-well-defined}. 
\end{proof}

\begin{definition}[{Cf.~\cref{def:Gh}}]\label{def:Gphi}
    For a set-valued homomorphism $\varphi\colon G\pto \Pi H$ with $\varphi\in \Hom(G,\Pi H)_f$, we define the digraph $\overrightarrow{G}_\varphi$ as $\overrightarrow{G}_h$, where $h\colon G\to \Pi H$ is a graph homomorphism satisfying $h\leq\varphi$ and $\norm{h}=\norm{\varphi}$; thanks to \cref{cor:Gh1=Gh2}, this definition does not depend on the choice of $h$.

    More explicitly, the set $V(\overrightarrow{G}_\varphi)$ of vertices of $\overrightarrow{G}_\varphi$ is given by 
    \[
    V(\overrightarrow{G}_\varphi)=\{\,u\in V(G)\mid \len(\varphi,u)\geq 2\,\},
    \]
    and the set $A(\overrightarrow{G}_\varphi)$ of arcs of $\overrightarrow{G}_\varphi$ is given by 
    \begin{multline*}
    A(\overrightarrow{G}_\varphi)=\{\,(u,v)\in V(\overrightarrow{G}_\varphi)\times V(\overrightarrow{G}_\varphi) \mid \text{$uv\in E(G)$ and for each $\xi\in \varphi(u)$ with $\len(\xi)=\len(\varphi,u)$}\\
    \text{and each $\eta\in \varphi(v)$ with $\len(\eta)=\len(\varphi,v)$, the type of $(\xi,\eta)$ is either (A1) or (A2)}\,\}.
    \end{multline*}
\end{definition}

Now we move to the proof of contractibility of $E_f$.
The basic strategy of the proof is as follows. We decompose $E_f$ as a chain of subposets 
\begin{equation}\label{eqn:Ef-as-union-of-Ys}
    \{\id_f\}=Y_0\subseteq Y_1\subseteq \cdots,\qquad E_f=\bigcup_{p\in \mathbb{N}} Y_p
\end{equation}
and show that $Y_p$ is a strong deformation retract of $Y_{p+1}$ for each $p\in\mathbb{N}$. This implies (by a standard argument as in \cite[Proof of Proposition~1A.1]{Hat02}) that $E_f$ is contractible. 
The decomposition \cref{eqn:Ef-as-union-of-Ys} is obtained via a few steps. 

Let $\{\,\omega_1,\omega_2,\dots, \omega_M\,\}$ be the set of all paths (i.e., (reduced) walks without repeated vertices) in $G$. Moreover, assume that we number the paths in $G$ so that the condition
\begin{equation}\label{eqn:numbering-condition}
    \text{for each $1\leq i\leq j\leq M$, we have $\len(\omega_i)\leq \len(\omega_j)$}
\end{equation}
is satisfied.

For each $n\in \mathbb{N}$, let $X_n$ be the subposet of $E_f$ defined by 
\[
X_n=\bigl\{\, \varphi\in E_f \;\big\vert\; \text{for each $u\in V(G)$, we have $\len(\varphi,u)\leq 2n$} \,\bigr\}.
\]
This defines a chain 
\begin{equation}\label{eqn:chain-Xn}
    X_0\subseteq X_1\subseteq \cdots,\qquad E_f=\bigcup_{n\in \mathbb{N}} X_n,
\end{equation}
and we have $X_0=\{\id_f\}$.

For each $n\in\mathbb{N}\setminus \{0\}$ and each $\varphi\in X_n$, let $\overrightarrow{G}_{\varphi,2n}$ be the induced subdigraph of $\overrightarrow{G}_\varphi$ consisting of all vertices $u\in V(G)$ with $\len(\varphi,u)=2n$. 
By \cref{lem:directed-cycle-is-f-tight} and \cref{prop:univ-cover-explicitly}, $\overrightarrow{G}_{\varphi,2n}$ does not contain any directed cycle. 
For each $i\in\mathbb{N}$ with $0\leq i\leq M$, let $X_{n,i}$ be the subposet of $X_n$ defined by 
\begin{equation*}
    X_{n,i}=\bigl\{\, \varphi\in X_{n} \;\big\vert\; \text{$\overrightarrow{G}_{\varphi,2n}$ contains none of $\omega_{i+1},\omega_{i+2},\dots,\omega_{M}$ as a directed walk} \,\bigr\}.
\end{equation*}
Notice that $\varphi\in X_{n,0}$ holds if and only if $\overrightarrow{G}_{\varphi,2n}$ is the empty digraph, i.e., $\varphi\in X_{n-1}$. 
Moreover, we have $X_{n,M}=X_n$.
Thus we have 
\begin{equation}\label{eqn:chain-Xnm}
    X_{n-1}=X_{n,0}\subseteq X_{n,1}\subseteq \cdots\subseteq X_{n,M}=X_n.
\end{equation}

The chain \cref{eqn:Ef-as-union-of-Ys} is obtained by refining the chain \cref{eqn:chain-Xn} by \cref{eqn:chain-Xnm}.
Therefore our goal is to show the following: for each $n\in \mathbb{N}\setminus\{0\}$ and each $i\in \mathbb{N}$ with $1\leq i\leq M$, the subposet $X_{n,i-1}$ of $X_{n,i}$ is a strong deformation retract of $X_{n,i}$.
We do so by applying \cref{prop:closure-sdr}. More precisely, we define a closure operator $U=U_{n,i}\colon X_{n,i}\to X_{n,i}$ and an interior operator $D=D_{n,i}\colon U(X_{n,i})\to U(X_{n,i})$ such that $DU(X_{n,i})=X_{n,i-1}$.

We move to the definitions of $U$ and $D$. Let $\omega_{i}=(v_0,v_1,\dots,v_k)$. Observe that, for any $\varphi\in X_{n,i}\setminus X_{n,i-1}$, the digraph $\overrightarrow{G}_{\varphi,2n}$ contains the path $\omega_i$ as a \emph{maximal} directed walk in $\overrightarrow{G}_{\varphi,2n}$ by \cref{eqn:numbering-condition}, and hence $v_k$ is a sink in $\overrightarrow{G}_{\varphi,2n}$. By \cref{lem:Gh-len}, $v_k$ is also a sink in $\overrightarrow{G}_\varphi$.

First, $U\colon X_{n,i}\to X_{n,i}$ is defined as follows.
For each $\varphi \in X_{n,i-1}\subseteq X_{n,i}$, we set $U(\varphi) = \varphi$. Suppose $\varphi \in X_{n,i}\setminus X_{n,i-1}$. 
Take $\xi\in \varphi(v_k)$ with $\len(\xi)=2n$, and let $\xi=(x_0,\dots, x_{2n})$. 
Define $\xi^\ast=(x_0,\dots, x_{2n-2})$, and 
define the function $U(\varphi)\colon V(G)\to \Pow\bigl(V(\Pi H)\bigr)\setminus\{\emptyset\}$ as 
\begin{equation}\label{eqn:def-of-U}
    U(\varphi)(u)=\begin{cases}
    \varphi(v_k)\cup \bigl\{\xi^\ast\bigr\} &\text{if $u=v_k$, and}\\
    \varphi(u) &\text{if $u\neq v_k$.}
\end{cases}
\end{equation}
It turns out that $\xi^\ast$ does not depend on the choice of $\xi\in \varphi(v_k)$; see \cref{lem:if-not-singleton-then-sink} and (a) below. Clearly, $\varphi \le U(\varphi)$ and $U^2(\varphi) = U(\varphi)$ hold. Hence, it needs to be shown that 
\begin{itemize}
    \item[(a)] $U(\varphi)$ is a set-valued homomorphism $G\pto \Pi H$, and
    \item[(b)] $U\colon X_{n,i}\to X_{n,i}$ is monotone.
\end{itemize}

Next we define the map $D \colon U(X_{n,i}) \to U(X_{n,i})$ as follows. If $\varphi \in X_{n,i-1}\subseteq U(X_{n,i})$, then we set $D(\varphi) = \varphi$. Suppose $\varphi \in U(X_{n,i})\setminus X_{n,i-1}$. Then $\omega_{i}$ is a directed walk in $\overrightarrow{G}_{\varphi,2n}$. 
By the definition of $U$, $\varphi(v_k)=U(\varphi)(v_k)$ contains $\xi^\ast$ as in \cref{eqn:def-of-U}. Define the function $D(\varphi)\colon V(G)\to \Pow\bigl(V(\Pi H)\bigr)\setminus \{\emptyset\}$ by 
\[
D(\varphi)(u)=\begin{cases}
    \{\xi^\ast\} &\text{if $u=v_k$, and}\\
    \varphi(u) &\text{if $u\neq v_k$.}
\end{cases}
\]
Then $D(\varphi)$ is a set-valued homomorphism $G\pto \Pi H$ because we have $D(\varphi)\leq\varphi$ and $\varphi\colon G\pto \Pi H$. Also, we have $\len\bigl(D(\varphi),v_k\bigr)=2n-2$ and hence $\overrightarrow{G}_{D(\varphi),2n}$ is the induced subdigraph of $\overrightarrow{G}_{\varphi,2n}$ obtained by deleting the vertex $v_k$. Therefore we have $D(\varphi)\in X_{n,i-1}$. 
Since $D^2(\varphi) = D(\varphi)$ clearly holds, it remains to show that
\begin{itemize}
    \item[(c)] $D\colon U(X_{n,i})\to U(X_{n,i})$ is monotone.
\end{itemize}

The rest of this section is devoted to proofs of the three remaining claims (a), (b), and (c). 

\begin{lemma}\label{lem:if-not-singleton-then-sink}
Let $\varphi\in E_f$.
Suppose that $u\in V(G)$ is a vertex such that $\varphi(u)$ is not a singleton. Then $u$ is a sink in $\overrightarrow{G}_\varphi$ and there exists at most one $\chi\in \varphi(u)$ with $\len(\chi)<\len(\varphi,u)$.
\end{lemma}
\begin{proof}
First observe that we have $\len(\varphi,u)\geq 2$ by \cref{prop:univ-cover-explicitly} (1), and hence $u\in V(\overrightarrow{G}_{\varphi})$.

Take $\xi,\xi'\in \varphi(u)$ with $\xi\neq\xi'$ and $\len(\xi)=\len(\varphi,u)$, and let $v\in V(G)$ be a vertex adjacent to $u$. Take $\eta\in \varphi(v)$ with $\len(\eta)=\len(\varphi,v)$. 
In order to show that $u$ is a sink in $\overrightarrow{G}_\varphi$, it suffices to show that the type of $(\xi,\eta)$ is neither (A1) nor (A2). 

If the type of $(\xi,\eta)$ is (A1), then we have $\len(\xi)+2=\len(\eta)$. Because $\len(\xi')\leq \len(\varphi,u)=\len(\xi)$, in order for $\xi'$ to be adjacent to $\eta$, we must have $\len(\xi')=\len(\xi)$ and the type of $(\xi',\eta)$ must also be (A1). However, this is impossible because we have $\xi\neq \xi'$. 

Suppose that the type of $(\xi,\eta)$ is (A2), and let $\ell=\len(\xi)=\len(\eta)\geq 2$. Let $\xi=(x_0,\dots,x_\ell)$ and $\eta=(y_0,\dots,y_\ell)$. Then we have $y_1=x_2\neq x_0=f(u)=s(\xi')$. 
By examining the possible types of $(\xi',\eta)$, we deduce a contradiction as follows. 
The type of $(\xi',\eta)$ cannot be (A5) since $\len(\eta)\geq 2$. Neither can it be (A1) nor (A4), by $y_1\neq s(\xi')$. The type cannot be (A3) because we have $\len(\xi')\leq \len(\xi)=\len(\eta)$. The remaining possibility is that the type of $(\xi',\eta)$ is also (A2), but this is impossible as well because we have $s(\xi)=f(u)=s(\xi')$ and $\xi\neq \xi'$.

We move to the proof of the second claim. Suppose to the contrary that we have $\xi,\xi',\xi''\in \varphi(u)$ with $\len(\xi)=\len(\varphi,u)$, $\len(\xi')<\len(\xi)$, $\len(\xi'')<\len(\xi)$, and $\xi'\neq \xi''$. Let $v\in V(G)$ be a vertex adjacent to $u$ and take $\eta\in \varphi(v)$; in fact, what we have just shown implies $\varphi(v)=\{\eta\}$, because $v$ is not a sink in $\overrightarrow{G}_\varphi$. Then $\eta$ must be adjacent to all of $\xi$, $\xi'$, and $\xi''$. 
By the above argument, the type of $(\xi,\eta)$ is either (A3) or (A4). 

Suppose that the type of $(\xi,\eta)$ is (A3). Then we have $\len(\xi')=\len(\xi'')=\len(\eta)$; let $\ell$ be the common length. By $\xi'\neq\xi''$ and $s(\xi')=f(u)=s(\xi'')$, we have $\ell\geq 2$. Thus the types of $(\xi',\eta)$ and $(\xi'',\eta)$ must both be (A2), but this is not possible. 

Finally, suppose that the type of $(\xi,\eta)$ is (A4). Then the types of $(\xi',\eta)$ and $(\xi'',\eta)$ must both be (A1), which is impossible.
\end{proof}

We show (a). Let $w\in V(G)$ be a vertex adjacent to $v_k$. Then by \cref{lem:if-not-singleton-then-sink}, $\varphi(w)$ is a singleton. Let us denote the unique element of $\varphi(w)$ by $\eta$. 
Take any $\xi\in\varphi(v_k)$ with $\len(\xi)=2n$. Then $\xi$ and $\eta$ are adjacent, and the type of $(\xi,\eta)$ is either (A3) or (A4) since $v_k$ is a sink in $\overrightarrow{G}_{\varphi}$.
It follows that $\xi^\ast$ as in \cref{eqn:def-of-U} and $\eta$ are also adjacent, the type of $(\xi^\ast,\eta)$ being either (A1), (A2), or (A5).
This completes the proof of (a).

\begin{lemma}\label{lem:Gphi2n-subgraph}
    Let $\varphi,\psi\in E_f$.
    Suppose that we have $\varphi\leq \psi$ and $\psi\in X_n$. Then $\overrightarrow{G}_{\varphi,2n}$ is an induced subdigraph of $\overrightarrow{G}_{\psi,2n}$, and all vertices in $V(\overrightarrow{G}_{\psi,2n})
    \setminus V(\overrightarrow{G}_{\varphi,2n})$ are sinks in $\overrightarrow{G}_{\psi}$ (and hence are also sinks in $\overrightarrow{G}_{\psi,2n}$).
\end{lemma}
\begin{proof}
    That $\overrightarrow{G}_{\varphi,2n}$ is an induced subdigraph of $\overrightarrow{G}_{\psi,2n}$ follows from \cref{def:Gphi}; we have $\overrightarrow{G}_{\varphi}=\overrightarrow{G}_h$ and $\overrightarrow{G}_\psi=\overrightarrow{G}_k$ for some graph homomorphisms $h,k\colon G\to \Pi H$ with $h(u)=k(u)$ for all $u\in V(\overrightarrow{G}_{\varphi,2n})$. All vertices in $V(\overrightarrow{G}_{\psi,2n})
    \setminus V(\overrightarrow{G}_{\varphi,2n})$ are sinks in $\overrightarrow{G}_{\psi}$ by \cref{lem:if-not-singleton-then-sink}.
\end{proof}

Now we can prove (b). 
Take $\varphi,\psi\in X_{n,i}$ with $\varphi\leq \psi$. If $U(\varphi)=\varphi$, then we have $U(\varphi)=\varphi\leq \psi\leq U(\psi)$. Therefore let us assume $\varphi < U(\varphi)$. Then we have $\varphi\notin X_{n,i-1}$.
It follows from $\varphi\leq \psi$ and \cref{lem:Gphi2n-subgraph} that $\overrightarrow{G}_{\varphi,2n}$ is an induced subdigraph of $\overrightarrow{G}_{\psi,2n}$. Therefore $\overrightarrow{G}_{\psi,2n}$ also contains the directed walk $\omega_{i}$, i.e., $\psi\notin X_{n,i-1}$. 
Thus we have 
\[
U(\psi)(v_k)=\psi(v_k)\cup\{\xi^\ast\}\supseteq \varphi(v_k)\cup\{\xi^\ast\}=U(\varphi)(v_k)
\]
in the notation of \cref{eqn:def-of-U}, and hence $U(\varphi)\leq U(\psi)$. 

Finally we show (c). Take $\varphi,\psi\in U(X_{n,i})$ with $\varphi\leq \psi$, and suppose that $D(\psi)<\psi$. This implies $\psi\in U(X_{n,i})\setminus X_{n,i-1}$, i.e., that the digraph $\overrightarrow{G}_{\psi,2n}$ contains $\omega_{i}$ as a maximal directed walk. 
Let $\xi^\ast$ be the reduced walk as in \cref{eqn:def-of-U}, defined using $\xi\in \psi(v_k)$ with $\len(\xi)=2n$.
We have $\xi^\ast\in \psi(v_k)$ by $\psi\in U(X_{n,i})$, and $\xi^\ast$ is the unique element of $\psi(v_k)$ satisfying $\len(\xi^\ast)<2n$ by \cref{lem:if-not-singleton-then-sink}. 
Since $v_k$ is a sink in $\overrightarrow{G}_{\psi,2n}$,
it follows from \cref{lem:Gphi2n-subgraph} that either 
\begin{itemize}
    \item $\overrightarrow{G}_{\varphi,2n}$ does not contain $v_k$, or
    \item $\overrightarrow{G}_{\varphi,2n}$ contains $\omega_{i}$ as a directed walk.
\end{itemize}
In the former case, we have $\len(\varphi,v_k)<2n$ and hence $\varphi(v_k)=\{\xi^\ast\}$ by $\varphi\leq \psi$. Therefore we have $D(\varphi)=\varphi\leq D(\psi)$. In the latter case, we have $\varphi\in U(X_{n,i})\setminus X_{n,i-1}$ and hence again $D(\varphi)\leq D(\psi)$.

This completes the proof of \cref{thm:contractible}. 
In particular, $E_f$ is simply-connected and hence $p_f\colon E_f\to B_f$ is the universal cover of $B_f$.

\section{The fundamental group of $B_f$}
\label{sec:fundamental-group}

The aim of this section is to determine the automorphism group $\Aut(p_f)$ of the universal cover $p_f\colon E_f\to B_f$ of $B_f$, where $f\colon G\to H$ is a homomorphism between graphs $G$ and $H$ satisfying the assumption of \cref{thm:main}. As explained in \cref{sec:intro}, this enables us to determine the homotopy type of $B_f$.

We begin with general observations. Thus in \cref{subsec:algebraic-str-HomGPiH,subsec:relation-to-Wrochna}, $G$ and $H$ denote general graphs. 

\subsection{An algebraic structure on $\Hom(G,\Pi H)$}\label{subsec:algebraic-str-HomGPiH}
We start by observing that the poset $\Hom(G,\Pi H)$ has a natural algebraic structure, given by the following ``partial multiplication.'' 
\begin{proposition}\label{prop:composition-in-HomGPiH}
    Let $G$ and $H$ be graphs, and $\varphi_1,\varphi_2\colon G\pto \Pi H$ set-valued homomorphisms with $t\varphi_1=s\varphi_2\colon G\pto H$.
    Then the function $\varphi_1\cdot\varphi_2\colon V(G)\to \Pow\bigl(V(\Pi H)\bigr)\setminus\{\emptyset\}$, defined as 
    \begin{equation}
    \label{eqn:composition-in-HomGPiH}
    (\varphi_1\cdot\varphi_2)(u)=\{\,\xi_1\cdot \xi_2\mid \text{$\xi_1\in \varphi_1(u)$, $\xi_2\in \varphi_2(u)$, and $t(\xi_1)=s(\xi_2)$}\,\}
    \end{equation}
    for any $u\in V(G)$, is a set-valued homomorphism $G\pto \Pi H$ satisfying $s(\varphi_1\cdot \varphi_2)=s\varphi_1$ and $t(\varphi_1\cdot\varphi_2)=t\varphi_2$.
\end{proposition}
\begin{proof}
    The condition $t\varphi_1=s\varphi_2$ ensures that $(\varphi_1\cdot \varphi_2)(u)$ is nonempty for any $u\in V(G)$ and that $s(\varphi_1\cdot \varphi_2)=s\varphi_1$ and $t(\varphi_1\cdot\varphi_2)=t\varphi_2$ hold. 
    Given a pair of adjacent vertices $u,v\in V(G)$, we need to show that any $\xi\in (\varphi_1\cdot\varphi_2)(u)$ and any $\eta\in (\varphi_1\cdot\varphi_2)(v)$ are adjacent in $\Pi H$. 
    By the definition of $\varphi_1\cdot\varphi_2$, we have $\xi=\xi_1\cdot \xi_2$ for some $\xi_1\in \varphi_1(u)$ and $\xi_2\in \varphi_2(u)$ with $t(\xi_1)=s(\xi_2)$, and similarly we have $\eta=\eta_1\cdot\eta_2$ for some
    $\eta_1\in \varphi_1(v)$ and $\eta_2\in\varphi_2(v)$ with $t(\eta_1)=s(\eta_2)$. 
    Since $\varphi_1$ is a set-valued homomorphism, $\xi_1$ and $\eta_1$ are adjacent in $\Pi H$, and similarly, since $\varphi_2$ is a set-valued homomorphism, $\xi_2$ and $\eta_2$ are adjacent in $\Pi H$. Therefore $\xi$ and $\eta$ are adjacent by \cref{prop:composition-on-PiH-is-graph-hom}.
\end{proof}

\begin{remark}
    Like \cref{rmk:PiH-as-internal-groupoid}, 
    this digression is intended for readers familiar with the notion of internal category \cite[Section~XII.1]{MacLane-CWM}. Let $G$ and $H$ be graphs.
    Consider the following pullback in the category of posets and monotone maps:
\[
\begin{tikzpicture}[baseline=-\the\dimexpr\fontdimen22\textfont2\relax ]
      \node(0) at (0,1) {$\Hom(G,\Pi H)\times_{\Hom(G,H)}\Hom(G,\Pi H)$};
      \node(1) at (5,1) {$\Hom(G,\Pi H)$};
      \node(2) at (0,-1) {$\Hom(G,\Pi H)$};
      \node(3) at (5,-1) {$\Hom(G,H)$.};
      \draw [->] (0) to node[auto,labelsize] {} (1);
      \draw [<-] (2) to node[auto,labelsize] {} (0);
      \draw [->] (1) to node[auto,labelsize] {$\Hom(G,s)$} (3);
      \draw [->] (2) to node[auto,labelsize] {$\Hom(G,t)$} (3);
      \draw (0.3,0.4) to (0.6,0.4) to (0.6,0.7);
\end{tikzpicture}
\]
    Explicitly, $\Hom(G,\Pi H)\times_{\Hom(G,H)}\Hom(G,\Pi H)$ is the subposet of $\Hom(G,\Pi H)\times\Hom(G,\Pi H)$ determined by the subset 
    \[
    \bigl\{\,(\varphi_1,\varphi_2)\in \Hom(G,\Pi H)\times\Hom(G,\Pi H) \;\big\vert\; t\varphi_1=s\varphi_2\,\bigr\}.
    \]
    \cref{prop:composition-in-HomGPiH} implies that we have a monotone map
    \[
    \widehat m\colon \Hom(G,\Pi H)\times_{\Hom(G,H)}\Hom(G,\Pi H)\to \Hom(G,\Pi H)
    \]
    mapping $(\varphi_1,\varphi_2)$ to $\varphi_1\cdot \varphi_2$. The tuple $\bigl(\Hom(G,s),\Hom(G,t),\Hom(G,e),\widehat m\bigr)$ of monotone maps (see \cref{def:s-t-i}) then makes $\Hom(G,\Pi H)$ into an internal category in the category of posets.
    In contrast to \cref{rmk:PiH-as-internal-groupoid}, $\Hom(G,\Pi H)$ is not an internal groupoid in general.
\end{remark}

\subsection{Relation to Wrochna's work}\label{subsec:relation-to-Wrochna}
It turns out that essentially all observations needed to determine the group $\Aut(p_f)$ are contained in \cite{Wro20}. In order to explain this, in this subsection we establish a precise relationship between the graph $\Pi H$ and Wrochna's work.

We start with the following observation.

\begin{proposition}
\label{cor:homotopy-walk}
    Let $G$ and $H$ be graphs, and $h\colon G\to \Pi H$ a graph homomorphism. Then for any walk $\omega$ in $G$ from $u$ to $v$, we have 
    \begin{equation*}
        h(v)=\overline{sh(\omega)}^{-1}\cdot h(u)\cdot \overline{th(\omega)}.
    \end{equation*}
\end{proposition}
\begin{proof}
    First suppose $\len(\omega)=1$, i.e., $\omega=(u,v)$. Then $u$ and $v$ are adjacent in $G$ and hence so are $h(u)$ and $h(v)$ in $\Pi H$. Therefore we have 
\[
h(v)=\bigl(sh(v),s h(u) \bigr)\cdot h(u)\cdot \bigl(t h(u),th(v) \bigr)
\]
    by \cref{def:adjacency-reduced-walks}.
    The general case can be proved by applying this repeatedly.
\end{proof}

\cref{cor:homotopy-walk} has the following consequence, relating the graph $\Pi H$ and \cite{Wro20}.
Recall from \cite[Definition~4.2]{Wro20} that, given graphs $G,H$, graph homomorphisms $f,g\colon G\to H$, and a vertex $u\in V(G)$, a reduced walk $\xi$ from $f(u)$ to $g(u)$ in $H$ is said to be \emph{topologically valid for $f,g,u$} if we have 
\[
\xi=\overline{f(\omega)}^{-1}\cdot \xi\cdot \overline{g(\omega)}
\]
for each closed walk $\omega$ on $u$ in $G$.
(In \cite{Wro20}, this definition is made under the assumption of \cref{thm:main}.)
Applying \cref{cor:homotopy-walk} to closed walks, we see that whenever $h\colon G\to \Pi H$ is a graph homomorphism and $u\in V(G)$, the reduced walk $h(u)$ in $H$ is topologically valid for $sh,th,u$. 
Recall from \cref{def:homotopy} the notion of homotopy between graph homomorphisms. It follows that any homotopy $h$ from $f$ to $g$ yields a topologically valid walk $h(u)$ for $f,g,u$, for any $u\in V(G)$.

\begin{proposition}
\label{prop:homotopy-top-valid-walk}
    Let $G$ and $H$ be graphs, and $f,g\colon G\to H$ graph homomorphisms. Suppose that $G$ is connected. Then, for each $u\in V(G)$, the function $h\mapsto h(u)$ gives rise to a bijection from the set of all homotopies from $f$ to $g$ to the set of all topologically valid walks for $f,g,u$.
\end{proposition}
\begin{proof}
    Since $G$ is connected, $h\mapsto h(u)$ is injective by \cref{cor:homotopy-walk}. To show that it is surjective, let $\xi$ be a topologically valid walk for $f,g,u$. For each $v\in V(G)$, choose any walk $\omega$ in $G$ from $u$ to $v$ and define 
    \[
    h(v)= \overline{f(\omega)}^{-1}\cdot \xi\cdot \overline{g(\omega)}.
    \]
    This definition is independent of the choice of $\omega$. Indeed, if $\omega$ and $\omega'$ are both walks in $G$ from $u$ to $v$, then $\omega\bullet \omega'^{-1}$ is a closed walk on $u$, and hence we have 
    \begin{align*}
        \xi
        &=\overline{f(\omega\bullet\omega'^{-1})}^{-1}\cdot \xi \cdot \overline{g(\omega\bullet\omega'^{-1})}\\
        &=\overline{f(\omega')}\cdot\overline{f(\omega)}^{-1}\cdot \xi \cdot \overline{g(\omega)}\cdot\overline{g(\omega')}^{-1},
    \end{align*}
    which implies 
    \[
    \overline{f(\omega')}^{-1}\cdot \xi\cdot \overline{g(\omega')}=\overline{f(\omega)}^{-1}\cdot \xi\cdot \overline{g(\omega)}.
    \]
    Thus we obtain a function $h\colon V(G)\to V(\Pi H)$. To see that $h$ is a graph homomorphism $G\to \Pi H$, suppose $vw\in E(G)$. Then, for any walk $\omega$ from $u$ to $v$, the concatenation $\omega \bullet(v,w)$ is a walk from $u$ to $w$.
    Using this walk in the definition of $h(w)$, we have 
    \[
    h(w)=\bigl(f(w),f(v)\bigr)\cdot h(v)\cdot \bigl(g(v),g(w)\bigr),
    \]
    and hence $h(v)$ and $h(w)$ are adjacent in $\Pi H$. $h$ is clearly a homotopy from $f$ to $g$.
\end{proof}

\subsection{The group $\Gamma_f$}
Throughout this subsection, let $G$ and $H$ be graphs satisfying the assumption of \cref{thm:main}, and let $f\colon G\to H$ be a graph homomorphism.
Consider the set 
\[
\Gamma_f=\{\,h\colon G\to \Pi H\mid \text{$sh=f=th$ and $h\in E_f$}\,\}
\]
of all homotopies $h\in E_f$ from $f$ to itself.
It is easy to see that this set becomes a group under the operation \cref{eqn:composition-in-HomGPiH}, in the light of \cref{prop:univ-cover-explicitly-atom}.

We define the map $\Gamma_f\times E_f\to E_f$ 
by sending $(h,\varphi)\in \Gamma_f\times E_f$ to $h\cdot\varphi$ as in \cref{eqn:composition-in-HomGPiH}, which is in $E_f$ by \cref{prop:univ-cover-explicitly,prop:composition-in-HomGPiH}.
Also, \cref{prop:composition-in-HomGPiH} ensures that we have $t(h\cdot \varphi)=t\varphi$, i.e., that the map $h\cdot(-)\colon E_f\to E_f$ makes the diagram
\[
\begin{tikzpicture}[baseline=-\the\dimexpr\fontdimen22\textfont2\relax ]
      \node(0) at (0,0.5) {$E_f$};
      \node(1) at (3,0.5) {$E_f$};
      \node(2) at (1.5,-0.5) {$B_f$};
      \draw [->] (0) to node[auto,labelsize] {$h\cdot(-)$} (1);
      \draw [<-] (2) to node[auto,labelsize] {$p_f$} (0);
      \draw [<-] (2) to node[auto,swap,labelsize] {$p_f$} (1);
\end{tikzpicture}
\]
commute. Therefore we obtain a group homomorphism $\Gamma_f\to \Aut(p_f)$. 
Moreover, one can identify the fiber $p_f^{-1}(f)\subseteq E_f$ of $p_f$ over $f\in B_f$ with $\Gamma_f$ by \cref{prop:HomGt-discrete-fibration}, and, under this identification, the restriction $\Gamma_f\times p_f^{-1}(f)\to p_f^{-1}(f)$ of the action $\Gamma_f\times E_f\to E_f$ to this fiber corresponds to the free and transitive action of $\Gamma_f$ on itself by left multiplication. In particular, $\Gamma_f$ acts on the fiber $p_f^{-1}(f)$ freely and transitively.
Therefore (see e.g.~\cite[Proposition~1.34]{Hat02} or \cite[Theorem~2.3.4]{Szamuely}) the group homomorphism $\Gamma_f\to \Aut(p_f)$ is an isomorphism.
Since the group $\Aut(p_f)$ is isomorphic to the fundamental group $\pi_1(B_f)$ (see e.g.~\cite[Proposition~1.39]{Hat02}), we have shown the following.

\begin{proposition}\label{prop:Lambda-is-fundamental-group}
    The fundamental group $\pi_1(B_f)$ of $B_f$ is isomorphic to $\Gamma_f$. 
\end{proposition}

Let us determine the group $\Gamma_f$.

\begin{definition}\label{def:piH}
    For each vertex $x\in V(H)$, we denote the group of all closed reduced (but not necessarily cyclically reduced) walks on $x$ by $\pi_1(H,x)$. The group operation on $\pi_1(H,x)$ is the product $\cdot$ of reduced walks. We write the subgroup of $\pi_1(H,x)$ consisting of all closed reduced walks of even lengths as $\pi_1(H,x)_\ev$; this subgroup is called the \emph{even part} of $\pi_1(H,x)$ in \cite[Section~3]{Mat17JMSUT}. 
    Since $H$ is connected, the groups $\pi_1(H,x)$ and $\pi_1(H,x)_\ev$ do not depend on $x\in V(H)$; at least they are well-defined up to non-canonical isomorphisms. Thus we will also write them simply as $\pi_1(H)$ and $\pi_1(H)_\ev$, respectively.
\end{definition}

It is well-known that $\pi_1(H)$ is isomorphic to the fundamental group of the geometric realization of the graph $H$ \cite[Corollary~3.6.17]{Spanier}.

\begin{proposition}\label{thm:Lambda-classification}
    The group $\Gamma_f$ is isomorphic to either of the following groups. 
    \begin{enumerate}[label=\emph{({\arabic*})}]
        \item The trivial group $\{1\}$.
        \item The infinite cyclic group $\mathbb{Z}$.
        \item The group $\pi_1(H)_\ev$.
    \end{enumerate}
\end{proposition}
\begin{proof}
    First we observe that, given any graph homomorphism $g\colon G\to H$ and any $u\in V(G)$, the bijection given in \cref{prop:homotopy-top-valid-walk} restricts to a bijection from the set of all homotopies from $f$ to $g$ \emph{in} $E_f$ to the set of all \emph{realizable} walks for $f,g,u$; the latter notion is introduced in \cite[Section~3]{Wro20}. For example, one can verify this by comparing the explicit description of $E_f$ in \cref{prop:univ-cover-explicitly-atom} and the characterization of realizable walks in \cite[Theorem~6.1]{Wro20}. 
    Therefore the set $\Gamma_f$ bijectively corresponds to the set of all realizable walks for $f,f,u$, which are certain closed reduced walks in $H$ on $f(u)$. 
    Moreover, under this identification, the product in $\Gamma_f$ corresponds to the product of realizable walks. The claim now follows from \cite[Theorem~8.1]{Wro20}.
\end{proof}

\begin{remark}\label{rmk:algo-for-Lambda}
    One can use an algorithm given in \cite{Wro20} to determine the group $\Gamma_f$ in polynomial time. More precisely, one can apply the algorithm described in \cite[Theorem~8.1]{Wro20} to the input data $G,H,f,f,u$, where $u$ is an arbitrary vertex of $G$.
    This, together with \cref{prop:Lambda-is-fundamental-group} and the fact that $B_f$ is aspherical (see \cref{sec:intro}), implies that one can use Wrochna's algorithm to determine the homotopy type of $B_f$. See the proof of \cref{thm:main} below for details.
\end{remark}

Let us consider the case (3) of \cref{thm:Lambda-classification} in more detail. 
It is easy to see the following (cf.~\cite[Section~3 and Proposition~6.12]{Mat17JMSUT}).
\begin{itemize}
    \item If $H$ is bipartite, then $\pi_1(H)_\ev=\pi_1(H)$.
    \item if $H$ is not bipartite, then $\pi_1(H)_\ev$ is a proper subgroup of $\pi_1(H)$ of index $2$. Moreover, $\pi_1(H)_\ev$ is isomorphic to $\pi_1(H\times K_2)$. 
\end{itemize}
Therefore $\pi_1(H)_\ev$ is a free group, and its rank $r$ is given by
\begin{equation}\label{eqn:rank-of-fundamental-group}
    r=\begin{cases}
        |E(H)|-|V(H)|+1 &\text{if $H$ is bipartite, and}\\
        2|E(H)|-2|V(H)|+1 &\text{if $H$ is not bipartite,}
    \end{cases}
\end{equation}
as can be seen by combining well-known facts (see e.g.~\cite[Corollary~3.7.5]{Spanier} and \cite[Propositions~3.8 and 3.9]{Lyndon-Schupp}).

We can now prove our main theorem and its refinement.

\begin{proof}[{Proof of \cref{thm:main}}]
    By \cref{thm:cover,thm:contractible}, the connected component $B_f$ of $\Hom(G,H)$ containing the given graph homomorphism $f\colon G\to H$ is the Eilenberg--MacLane space $K\bigl(\pi_1(B_f),1\bigr)$, or equivalently (by \cref{prop:Lambda-is-fundamental-group}), $K(\Gamma_f,1)$. Since an Eilenberg--MacLane space is unique up to homotopy equivalence (see e.g.~\cite[Theorem~1B.8]{Hat02}), we have the following.
    \begin{itemize}
        \item[(1)] If (1) of \cref{thm:Lambda-classification} happens, then $B_f$ is $K(\{1\},1)$, and hence is homotopy equivalent to a point.
        \item[(2)] If (2) of \cref{thm:Lambda-classification} happens, then $B_f$ is $K(\mathbb{Z},1)$, and hence is homotopy equivalent to a circle. 
        \item[(3a)] If (3) of \cref{thm:Lambda-classification} happens and $H$ is bipartite, then $B_f$ is $K\bigl(\pi_1(H),1\bigr)$, and hence is homotopy equivalent to (the geometric realization of) $H$. 
        \item[(3b)] If (3) of \cref{thm:Lambda-classification} happens and $H$ is not bipartite, then $B_f$ is $K\bigl(\pi_1(H\times K_2),1\bigr)$, and hence is homotopy equivalent to (the geometric realization of) $H\times K_2$.\qedhere
    \end{itemize}
\end{proof}

\begin{proof}[{Proof of \cref{prop:main-cases}}]
    Looking more closely at \cite[Proof of Theorem~8.1]{Wro20}, we see that case~3 of \cite[Theorem~8.1]{Wro20} (corresponding to (3) of \cref{thm:Lambda-classification} and \cref{thm:main}) occurs if and only if case~3 of \cite[Lemma~7.5]{Wro20} occurs. 
    Now, \cite[Proof of Lemma~7.5]{Wro20} reveals that case~3 of \cite[Lemma~7.5]{Wro20} occurs if and only if $G$ is bipartite. Thus if $G$ is not bipartite, only (1) or (2) of \cref{thm:Lambda-classification} and \cref{thm:main} occurs; and, of course, if moreover $H$ is bipartite, then there is no graph homomorphism $G\to H$ and hence $\Hom(G,H)$ is empty. 

    Suppose that $G$ is bipartite. Then there exists a graph homomorphism $g\colon G\to K_2$. Since $H$ has at least one edge, we also have a graph homomorphism $j\colon K_2\to H$. The composite graph homomorphism $jg\colon G\to H$ then satisfies case~3 of \cite[Lemma~7.5]{Wro20} (or more precisely, the pair $jg,jg$ satisfies it, because \cite[Lemma~7.5]{Wro20} deals with a pair of homomorphisms). 
    It is easy to see that any two homomorphisms $G\to H$ factoring through $K_2$ are in the same connected component of $\Hom(G,H)$ if and only if $H$ is not bipartite. 
\end{proof}

\appendix
\section{An approach via the covering theory of graphs}
\label{apx:comparing}
The aim of this appendix is to describe another approach to construct the universal cover $p_f\colon E_f\to B_f$. 
This approach is based on the covering theory of graphs, as recalled in \cref{subsec:univ-cover-graph,subsec:Galois}, and gives rise to a \emph{family} of definitions of the universal cover $p_f\colon E_f\to B_f$. 
We explain this and establish its equivalence with our definition given in \cref{subsec:def-of-pf} (see \cref{thm:univ-cov-Galois}). 
Matsushita's definition mentioned in \cref{sec:intro} turns out to be a member of this family (see \cref{ex:cover-approach}).

\subsection{Group actions on graphs}\label{subsec:group-action}
We begin with a quick review of group actions on graphs. Let $G$ be a graph. An \emph{automorphism on $G$} is a graph isomorphism $G\to G$. We denote by $\Aut(G)$ be the group of all automorphisms on $G$.

Let $\Gamma$ be a group.
A \emph{$\Gamma$-graph} is a graph $G$ equipped with a \emph{$\Gamma$-action} on it, by which we mean a group homomorphism $h\colon \Gamma\to \Aut(G)$. The group homomorphism $h$ can be equivalently specified by the function $\Gamma\times V(G)\to V(G)$ mapping each $(\gamma,v)\in \Gamma\times V(G)$ to $\bigl(h(\gamma)\bigr)(v)\in V(G)$. 
We often write $\bigl(h(\gamma)\bigr)(v)$ as $\gamma\cdot v$.
Notice that we also use the symbol $\cdot$ to denote the product of reduced walks (see \cref{subsec:walk}). 
As it turns out, in the specific $\Gamma$-graphs we encounter in this appendix, the group actions are indeed often given by (or at least induced by) the product of reduced walks.

Let $G$ and $H$ be $\Gamma$-graphs. 
Then a graph homomorphism $f\colon G\to H$ is said to be \emph{$\Gamma$-equivariant} if for each $\gamma\in \Gamma$ and $v\in V(G)$, we have $\gamma\cdot f(v)=f(\gamma\cdot v)$. Similarly, a set-valued homomorphism $\varphi\colon G\pto H$ is \emph{$\Gamma$-equivariant} if for each $\gamma\in \Gamma$ and $v\in V(G)$, we have $\gamma\cdot \varphi(v)=\varphi(\gamma\cdot v)$, where the left-hand side $\gamma\cdot \varphi(v)$ denotes the set 
\[
\{\,\gamma\cdot x \mid x\in \varphi(v)\,\}\subseteq V(H).
\]
We write the subposet of $\Hom(G,H)$ consisting of all $\Gamma$-equivariant set-valued homomorphisms $G\pto H$ by $\Hom^\Gamma(G,H)$.

Let $k\colon \Gamma\to \Gamma'$ be a group homomorphism. Then any $\Gamma'$-graph $\bigl(G,h\colon \Gamma'\to\Aut(G)\bigr)$ gives rise to the $\Gamma$-graph $\bigl(G,hk\colon \Gamma\to \Gamma'\to\Aut(G)\bigr)$.

\subsection{The universal cover of a graph}\label{subsec:univ-cover-graph}
Here we recall the basics of the covering theory for graphs. See e.g.~\cite[Sections~1.3 and 1.A]{Hat02}. 

Recall the group $\pi_1(H,x)$ defined in \cref{def:piH} for a graph $H$ and a vertex $x\in V(H)$. If $f\colon G\to H$ is a graph homomorphism and $u\in V(G)$ and $x\in V(H)$ are vertices with $f(u)=x$, then we obtain a group homomorphism $f_\ast\colon \pi_1(G,u)\to \pi_1(H,x)$ mapping each $\omega\in \pi_1(G,u)$ to $f_\ast(\omega)=\overline{f(\omega)}$. By \cref{prop:covering-preserves-reduced-walks}, if $f$ is a covering map of graphs, then we have $f_\ast(\omega)=f(\omega)$, and it is easy to see that the group homomorphism $f_\ast\colon \pi_1(G,u)\to \pi_1(H,x)$ is injective in this case.

For any graph $H$ and any vertex $x\in V(H)$, we define
\[
V(\Pi H)_x=\{\,\xi\in V(\Pi H)\mid s(\xi)=x\,\}. 
\]

Let $H$ be a connected graph and $x\in V(H)$.
Define the graph $\widetilde H$ by $V(\widetilde H) = V(\Pi H)_x$ and, for each $\xi, \eta\in V(\widetilde H)$, $\xi\eta\in E(\widetilde H)$ if and only if $t(\xi)$ and $t(\eta)$ are adjacent in $H$ and $\xi=\eta\cdot \bigl(t(\eta),t(\xi)\bigr)$ holds (cf.~\cref{def:adjacency-reduced-walks}).
Define the graph homomorphism $p_H\colon \widetilde H\to H$ by mapping each $\xi\in V(\widetilde H)$ to $p_H(\xi)=t(\xi)\in V(H)$; it is called the \emph{universal cover} of $H$.
It is easy to see that $p_H\colon \widetilde H\to H$ is a covering map of graphs and its domain $\widetilde H$ is a tree (i.e., a connected graph without cycles).
(In fact, these properties characterize $p_H\colon \widetilde H\to H$ up to isomorphism, although we will not need this fact. In particular, the definition of $p_H\colon \widetilde H\to H$ does not depend on the choice of the vertex $x$.)

The graph $\widetilde H$ has a natural $\pi_1(H,x)$-graph structure, given by the function $\pi_1(H,x)\times V(\widetilde H)\to V(\widetilde H)$ mapping $(\xi,\eta)\in \pi_1(H,x)\times V(\widetilde H)$ to $\xi\cdot \eta\in V(\widetilde H)$, where $\cdot$ denotes the product of reduced walks (see \cref{subsec:walk}). 
For each $\xi\in \pi_1(H,x)$, the diagram 
    \[
\begin{tikzpicture}[baseline=-\the\dimexpr\fontdimen22\textfont2\relax ]
      \node(0) at (0,0.5) {$\widetilde H$};
      \node(1) at (3,0.5) {$\widetilde H$};
      \node(2) at (1.5,-0.5) {$H$};
      \draw [->] (0) to node[auto,labelsize] {$\xi\cdot(-)$} (1);
      \draw [<-] (2) to node[auto,labelsize] {$p_H$} (0);
      \draw [<-] (2) to node[auto,swap,labelsize] {$p_H$} (1);
\end{tikzpicture}
\]
commutes.

Let $G$ and $H$ be connected graphs and $f\colon G\to H$ a graph homomorphism. Choose vertices $u\in V(G)$ and $x\in V(H)$ with $f(u)=x$, and construct the universal covers $p_G\colon \widetilde G\to G$ and $p_H\colon \widetilde H\to H$ with $V(\widetilde G)=V(\Pi G)_u$ and $V(\widetilde H)=V(\Pi H)_x$. Then we obtain a graph homomorphism $\widetilde f\colon \widetilde G\to \widetilde H$ mapping each $\omega\in V(\widetilde G)$ to $\widetilde f(\omega)=\overline{f(\omega)}\in V(\widetilde H)$. This graph homomorphism makes the diagram 
\begin{equation}\label{eqn:f-tilde}
\begin{tikzpicture}[baseline=-\the\dimexpr\fontdimen22\textfont2\relax ]
      \node(0) at (0,1) {$\widetilde G$};
      \node(1) at (3,1) {$\widetilde H$};
      \node(2) at (0,-1) {$G$};
      \node(3) at (3,-1) {$H$};
      \draw [->] (0) to node[auto,labelsize] {$\widetilde f$} (1);
      \draw [<-] (2) to node[auto,labelsize] {$p_G$} (0);
      \draw [->] (1) to node[auto,labelsize] {$p_H$} (3);
      \draw [->] (2) to node[auto,labelsize] {$f$} (3);
\end{tikzpicture}   
\end{equation}
commute. 
Let $\Gamma=\pi_1(G,u)$. 
By the above discussion, $\widetilde G$ has a natural $\Gamma$-graph structure; similarly, $\widetilde H$ has a natural $\pi_1(H,x)$-graph structure.
The group homomorphism 
$f_\ast\colon \Gamma\to \pi_1(H,x)$ makes $\widetilde H$ into a $\Gamma$-graph as well (see the final paragraph of \cref{subsec:group-action}).
It is easy to see that $\widetilde f$ is a $\Gamma$-equivariant graph homomorphism, i.e., $\widetilde f\in \Hom^\Gamma(\widetilde G,\widetilde H)$.
We have a monotone map $\beta\colon \Hom^\Gamma(\widetilde G,\widetilde H)\to \Hom(G,H)$ mapping each $\psi\in \Hom^\Gamma(\widetilde G,\widetilde H)$ to $\beta(\psi)\colon G\pto H$ defined as follows. 
For each $v\in V(G)$, choose $\widetilde v\in V(\widetilde G)$ with $p_G(\widetilde v)=v$, and define $\beta(\psi)(v)=\{\,p_H(\widetilde y)\mid \widetilde y\in \psi(\widetilde v)\,\}$.
This definition does not depend on the choice of $\widetilde v\in p_G^{-1}(v)$ because $\Gamma$ acts transitively on $p_G^{-1}(v)$ and $\psi$ is $\Gamma$-equivariant. 

Now we can state the main result of this appendix.
\begin{theorem}\label{thm:univ-cov-via-tilde}
    Let $G$ and $H$ be connected graphs and $f\colon G\to H$ a graph homomorphism. Let $p_G\colon \widetilde G\to G$ and $p_H\colon \widetilde H\to H$ be the universal covers. Let $\Gamma=\pi_1(G)$, and make $\widetilde G$ and $\widetilde H$ into $\Gamma$-graphs as above. 
    Then there exists a poset isomorphism $\Psi\colon \Hom(G,\Pi H)_f\to \Hom^\Gamma(\widetilde G,\widetilde H)$ making the diagram 
    \begin{equation}\label{eqn:Psi-commutativity}
\begin{tikzpicture}[baseline=-\the\dimexpr\fontdimen22\textfont2\relax ]
      \node(0) at (0,1) {$\Hom(G,\Pi H)_f$};
      \node(1) at (4,1) {$\Hom^\Gamma(\widetilde G,\widetilde H)$};
      \node(2) at (2,-1) {$\Hom(G,H)$};
      \draw [->] (0) to node[auto,labelsize] {$\Psi$} (1);
      \draw [<-] (2) to node[auto,labelsize] {$\Hom(G,t)|_{\Hom(G,\Pi H)_f}$} (0);
      \draw [<-] (2) to node[auto,swap,labelsize] {$\beta$} (1);
\end{tikzpicture}
    \end{equation}
    commute. 
\end{theorem}

\subsection{Properties of $\Pi H$}\label{subsec:properties-of-Pi}
In order to prove \cref{thm:univ-cov-via-tilde}, in this subsection we establish some properties of graphs of the form $\Pi H$. 

\begin{proposition}\label{prop:tree-PiH}
    Let $H$ be a tree. Then the function $\langle s,t\rangle\colon V(\Pi H)\to V(H)\times V(H)$ is a bijection.
    In particular, for any $x\in V(H)$, the map $t\colon V(\Pi H)_x\to V(H)$ is a bijection.
\end{proposition}
\begin{proof}
    In the tree $H$, given any two vertices $x,y\in V(H)$, there exists a unique reduced walk from $x$ to $y$ in $H$. Therefore $\langle s,t\rangle\colon V(\Pi H)\to V(H\times H)$ is a bijection. 
    The second assertion follows from the first by restriction. 
\end{proof}

\begin{corollary}\label{cor:st-iso}
    Let $H$ be a tree. Then the graph homomorphism $\langle s,t\rangle\colon \Pi H\to H\times H$ is a graph isomorphism.
\end{corollary}
\begin{proof}
    The claim follows from \cref{prop:tree-PiH,prop:st-covering} because any bijective covering map of graphs is a graph isomorphism. 
\end{proof}

For any graph homomorphism $p\colon G\to H$, let $\Pi p\colon \Pi G\to \Pi H$ be the graph homomorphism defined by $(\Pi p)(\omega)=\overline{p(\omega)}$ for each $\omega\in V(\Pi G)$ (see \cref{rmk:adjacency-between-non-reduced-walks}). When $p$ is a covering map of graphs, we have $(\Pi p)(\omega)=p(\omega)$ for each $\omega\in V(\Pi G)$ by \cref{prop:covering-preserves-reduced-walks}.

\begin{proposition}\label{prop:s-cover-cartesian}
    Let $p\colon G\to H$ be a covering map of graphs. Then the commutative square 
\begin{equation}\label{eqn:s-cover-cartesian}
\begin{tikzpicture}[baseline=-\the\dimexpr\fontdimen22\textfont2\relax ]
      \node(0) at (0,1) {$\Pi G$};
      \node(1) at (3,1) {$\Pi H$};
      \node(2) at (0,-1) {$G$};
      \node(3) at (3,-1) {$H$};
      \draw [->] (0) to node[auto,labelsize] {$\Pi p$} (1);
      \draw [<-] (2) to node[auto,labelsize] {$s$} (0);
      \draw [->] (1) to node[auto,labelsize] {$s$} (3);
      \draw [->] (2) to node[auto,labelsize] {$p$} (3);
\end{tikzpicture}
\end{equation}
    is a pullback of graphs. In other words, the graph homomorphism $\langle s,\Pi p\rangle \colon \Pi G\to G\times \Pi H$, mapping $\omega\in V(\Pi G)$ to $\bigl(s(\omega),p(\omega)\bigr)\in V(G)\times V(\Pi H)$, induces an isomorphism of graphs between $\Pi G$ and the induced subgraph $G\times_H \Pi H$ of $G\times \Pi H$ determined by the subset
    \[
    V(G\times_H \Pi H)=\bigl\{\,(v,\xi)\in V(G)\times V(\Pi H)\;\vert\; p(v)=s(\xi)\,\bigr\}
    \]
    of $V(G\times \Pi H)$.
\end{proposition}
\begin{proof}
    Given $v\in V(G)$ and $\xi\in V(\Pi H)$ with $p(v)=s(\xi)$, there exists a unique $\widehat\xi\in V(\Pi G)$ with $s(\widehat \xi)=v$ and $(\Pi p)(\widehat \xi)=\xi$, as can be easily seen by induction on $\len(\xi)$. It remains to show the following: given $v,w\in V(G)$ and $\xi,\eta\in V(\Pi H)$ such that $p(v)=s(\xi)$ and $p(w)=s(\eta)$, if $vw\in E(G)$ and $\xi\eta\in E(\Pi H)$, then $\widehat \xi\widehat \eta\in E(\Pi G)$, where $\widehat \xi,\widehat \eta\in V(\Pi G)$ are the unique elements such that $s(\widehat \xi)=v$, $(\Pi p)(\widehat \xi)=\xi$, $s(\widehat \eta)=w$, and $(\Pi p)(\widehat \eta)=\eta$. 
    This can be seen by induction on $\len(\xi)+\len(\eta)$.
\end{proof}

\begin{corollary}
\label{prop:covering-Pi}
    Let $p\colon G\to H$ be a covering map of graphs and $v\in V(G)$. Then the map $\Pi p\colon V(\Pi G)_v\to V(\Pi H)_{p(v)}$, defined by mapping $\omega\in V(\Pi G)_v$ to $p(\omega)\in V(\Pi H)_{p(v)}$, is a bijection.
\end{corollary}
\begin{proof}
    The pullback square \cref{eqn:s-cover-cartesian} induces the stated bijection between the fibers. Alternatively, see the first sentence of the proof of \cref{prop:s-cover-cartesian}.
\end{proof}

\cref{prop:covering-Pi} is an analogue of the unique path lifting property of covering maps of topological spaces.

\begin{proposition}
\label{prop:proj-adj}
    Let $H$ be a connected graph, $p\colon \widetilde H\to H$ its universal cover, and $\widetilde x_1,\widetilde x_2$ adjacent vertices in $\widetilde H$. 
    Then, for any $\widetilde \xi_1\in V(\Pi \widetilde H)_{\widetilde x_1}$ and $\widetilde \xi_2\in V(\Pi \widetilde H)_{\widetilde x_2}$, the following conditions are equivalent.
    \begin{enumerate}[label=\emph{({\arabic*})}]
        \item $\widetilde\xi_1$ and $\widetilde\xi_2$ are adjacent in $\Pi \widetilde H$.
        \item $t(\widetilde \xi_1)$ and $t(\widetilde \xi_2)$ are adjacent in $\widetilde H$. 
        \item $(\Pi p)(\widetilde \xi_1)=p(\widetilde \xi_1)$ and $(\Pi p)(\widetilde \xi_2)=p(\widetilde \xi_2)$ are adjacent in $\Pi H$.
    \end{enumerate}
\end{proposition}
\begin{proof}
    (1) and (2) are equivalent by \cref{cor:st-iso} since $\widetilde H$ is a tree.
    (1) and (3) are equivalent by \cref{prop:s-cover-cartesian} since $p\colon \widetilde H\to H$ is a covering map of graphs. 
\end{proof}

Let $H$ be a connected graph and $p\colon \widetilde H\to H$ its universal cover. 
Take any vertex $\widetilde x\in V(\widetilde H)$. 
Since $\widetilde H$ is a tree, we have a bijection 
$t\colon V(\Pi \widetilde H)_{\widetilde x}\to V(\widetilde H)$ by \cref{prop:tree-PiH}. Since $p\colon\widetilde H\to H$ is a covering map, we have a bijection $\Pi p\colon V(\Pi \widetilde H)_{\widetilde x}\to V(\Pi H)_{p(\widetilde x)}$ by \cref{prop:covering-Pi}. 
We define $\lift_{\widetilde x}\colon V(\Pi H)_{p(\widetilde x)}\to V(\widetilde H)$ to be the composite of 
\[
\begin{tikzpicture}[baseline=-\the\dimexpr\fontdimen22\textfont2\relax ]
      \node(0) at (3,0) {$V(\Pi\widetilde H)_{\widetilde x}$};
      \node(1) at (0,0) {$V(\Pi H)_{p(\widetilde x)}$};
      \node(2) at (6,0) {$V(\widetilde H)$};
      \draw [<-] (0) to node[auto,swap,labelsize] {$(\Pi p)^{-1}$} node[auto,labelsize] {$\cong$} (1);
      \draw [<-] (2) to node[auto,swap,labelsize] {$t$} node[auto,labelsize] {$\cong$}  (0);
\end{tikzpicture}
\]
and $\proj_{\widetilde x}\colon  V(\widetilde H)\to V(\Pi H)_{p(\widetilde x)}$ to be 
the composite of 
\[
\begin{tikzpicture}[baseline=-\the\dimexpr\fontdimen22\textfont2\relax ]
      \node(0) at (3,0) {$V(\Pi\widetilde H)_{\widetilde x}$};
      \node(1) at (6,0) {$V(\Pi H)_{p(\widetilde x)}$.};
      \node(2) at (0,0) {$V(\widetilde H)$};
      \draw [->] (0) to node[auto,labelsize] {$\Pi p$} node[auto,swap,labelsize] {$\cong$} (1);
      \draw [->] (2) to node[auto,labelsize] {$t^{-1}$} node[auto,swap,labelsize] {$\cong$}  (0);
\end{tikzpicture}
\]
Clearly $\lift_{\widetilde x}$ and $\proj_{\widetilde x}$ are inverse to each other. 
Moreover, since the diagram 
\[
\begin{tikzpicture}[baseline=-\the\dimexpr\fontdimen22\textfont2\relax ]
      \node(0) at (0,1) {$V(\Pi\widetilde H)_{\widetilde x}$};
      \node(1) at (3,1) {$V(\Pi H)_{p(\widetilde x)}$};
      \node(2) at (0,-1) {$V(\widetilde H)$};
      \node(3) at (3,-1) {$V(H)$};
      \draw [->] (0) to node[auto,labelsize] {$\Pi p$} node[auto,swap,labelsize] {$\cong$} (1);
      \draw [<-] (2) to node[auto,labelsize] {$t$} node[auto,swap,labelsize] {$\cong$}  (0);
      \draw [->] (1) to node[auto,labelsize] {$t$} (3);
      \draw [->] (2) to node[auto,labelsize] {$p$} (3);
\end{tikzpicture}
\]
commutes, we have 
\begin{equation}\label{eqn:lift-proj-fibred}
    p\bigl(\lift_{\widetilde x}(\xi)\bigr)=t(\xi)
    \quad\text{and}\quad
    t\bigl(\proj_{\widetilde x}(\widetilde y)\bigr)=p(\widetilde y)
\end{equation}
for any $\xi\in V(\Pi H)_{p_H(\widetilde x)}$ and $\widetilde y\in V(\widetilde H)$.

The following is an immediate consequence of \cref{prop:proj-adj}.
\begin{corollary}
\label{prop:lift-proj-pres-adj}
    Let $H$ be a connected graph, $p\colon \widetilde H\to H$ its universal cover, and $\widetilde x_1,\widetilde x_2$ adjacent vertices in $\widetilde H$. 
    \begin{enumerate}[label=\emph{({\arabic*})}]
        \item Let $\xi_1\in V(\Pi H)_{p(\widetilde x_1)}$ and $\xi_2\in V(\Pi H)_{p(\widetilde x_2)}$. If $\xi_1$ and $\xi_2$ are adjacent in $\Pi H$, then $\lift_{\widetilde x_1}(\xi_1)$ and $\lift_{\widetilde x_2}(\xi_2)$ are adjacent in $\widetilde H$.
        \item Let $\widetilde y_1\in V(\widetilde H)$ and $\widetilde y_2\in V(\widetilde H)$. If $\widetilde y_1$ and $\widetilde y_2$ are adjacent in $\widetilde H$, then $\proj_{\widetilde x_1}(\widetilde y_1)$ and $\proj_{\widetilde x_2}(\widetilde y_2)$ are adjacent in $\Pi H$.
    \end{enumerate}
\end{corollary}

\begin{proposition}\label{prop:lift-proj-equivariance}
    Let $H$ be a connected graph, $p\colon\widetilde H\to H$ its universal cover, and $\gamma\colon \widetilde H\to \widetilde H$ a graph homomorphism making the diagram
    \[
\begin{tikzpicture}[baseline=-\the\dimexpr\fontdimen22\textfont2\relax ]
      \node(0) at (0,0.5) {$\widetilde H$};
      \node(1) at (3,0.5) {$\widetilde H$};
      \node(2) at (1.5,-0.5) {$H$};
      \draw [->] (0) to node[auto,labelsize] {$\gamma$} (1);
      \draw [<-] (2) to node[auto,labelsize] {$p$} (0);
      \draw [<-] (2) to node[auto,swap,labelsize] {$p$} (1);
\end{tikzpicture}
\]
    commute.
    Then for any $\widetilde x,\widetilde y\in V(\widetilde H)$ and $\xi\in V(\Pi H)_{p(\widetilde x)}$,
    we have $\gamma \bigl(\lift_{\widetilde x}(\xi)\bigr)=\lift_{\gamma(\widetilde x)}(\xi)$ and $\proj_{\widetilde x}(\widetilde y)=\proj_{\gamma(\widetilde x)}\bigl(\gamma(\widetilde y)\bigr)$.
\end{proposition}
\begin{proof}
    These follow from the commutativity of the following diagram: 
\[
\begin{tikzpicture}[baseline=-\the\dimexpr\fontdimen22\textfont2\relax ]
      \node(0) at (0,3.5) {$V(\Pi\widetilde H)_{\widetilde x}$};
      \node(1) at (6,3.5) {$V(\Pi H)_{p(\widetilde x)}$};
      \node(2) at (0,1) {$V(\widetilde H)$};
      \node(3) at (6,1) {$V(H)$.};
      \node(4) at (3,2.5) {$V(\Pi \widetilde H)_{\gamma(\widetilde x)}$};
      \node(5) at (3,0) {$V(\widetilde H)$};
      \draw [->] (0) to node[auto,labelsize] {$\Pi p$} node[auto,swap,labelsize] {$\cong$} (1);
      \draw [<-] (2) to node[auto,labelsize] {$t$} node[auto,swap,labelsize] {$\cong$} (0);
      \draw [->] (1) to node[auto,labelsize] {$t$} (3);
      \draw [->] (2) to node[auto,near end,labelsize] {$p$} (3);
      \draw [->] (0) to node[auto,swap,labelsize] {$\Pi\gamma$} (4);
      \draw [->] (2) to node[auto,swap,labelsize] {$\gamma$} (5);
      \draw [->] (4) to node[auto,labelsize] {$\Pi p$} node[auto,swap,labelsize] {$\cong$} (1);
      \draw [->] (5) to node[auto,swap,labelsize] {$p$} (3);
      \draw[white,-,line width=4] (4) to (5);
      \draw [->] (4) to node[auto,near start,swap,labelsize] {$t$} node[auto,near start,labelsize] {$\cong$} (5);
\end{tikzpicture}\qedhere
\]
\end{proof}

\subsection{The proof of \cref{thm:univ-cov-via-tilde}}
Now we are ready to prove \cref{thm:univ-cov-via-tilde}. Throughout this subsection, let $G$ and $H$ be connected graphs and $f\colon G\to H$ a graph homomorphism. We fix the universal covers $p_G\colon \widetilde G\to G$ and $p_H\colon \widetilde H\to H$.
Let $\Gamma=\pi_1(G)$, make $\widetilde G$ and $\widetilde H$ into $\Gamma$-graphs as in \cref{subsec:univ-cover-graph}, and choose a $\Gamma$-equivairant graph homomorphism $\widetilde f\colon \widetilde G\to \widetilde H$ making the diagram \cref{eqn:f-tilde} commute. 

We define the monotone map $\Psi\colon \Hom(G,\Pi H)_f\to \Hom^\Gamma(\widetilde G,\widetilde H)$ appearing in the statement of \cref{thm:univ-cov-via-tilde} as follows.
Given any $\varphi\in \Hom(G,\Pi H)_f$, we define the function $\Psi(\varphi)\colon V(\widetilde G)\to \Pow\bigl(V(\widetilde H)\bigr)\setminus \{\emptyset\}$ by mapping each vertex $\widetilde v\in V(\widetilde G)$ to $\lift_{\widetilde f(\widetilde v)}\bigl(\varphi(p_G(\widetilde v))\bigr)= \{\,\lift_{\widetilde f(\widetilde v)}(\xi)\mid \xi\in \varphi(p_G(\widetilde v))\,\}$, which makes sense thanks to $s\varphi=f$ and \cref{eqn:f-tilde}. 
By \cref{prop:lift-proj-pres-adj} (1), we see that $\Psi(\varphi)$ is a set-valued homomorphism $\widetilde G\pto \widetilde H$. 
By \cref{prop:lift-proj-equivariance} and the $\Gamma$-equivariance of $\widetilde f$, we see that $\Psi(\varphi)$ is $\Gamma$-equivariant. 
Thus we obtain a monotone map $\Psi\colon \Hom(G,\Pi H)_f\to \Hom^\Gamma(\widetilde G,\widetilde H)$, which makes the diagram \cref{eqn:Psi-commutativity} commute by \cref{eqn:lift-proj-fibred}. 

Next we define a monotone map $\Phi\colon \Hom^\Gamma(\widetilde G,\widetilde H)\to \Hom(G,\Pi H)_f$, which will turn out to be the inverse of $\Psi$. 
Given any $\Gamma$-equivariant set-valued homomorphism $\psi\colon \widetilde G\pto \widetilde H$, the function  $\Phi(\psi)\colon V(G)\to \Pow\bigl(V(\Pi H)\bigr)\setminus\{\emptyset\}$ is defined as follows. Given $v\in V(G)$, we first choose $\widetilde v\in p_G^{-1}(v)\subseteq V(\widetilde G)$ and set $\Phi(\psi)(v)= \proj_{\widetilde f(\widetilde v)}\bigl(\psi(\widetilde v)\bigr)=\{\,\proj_{\widetilde f(\widetilde v)}(\widetilde x)\mid \widetilde x\in \psi(\widetilde v)\,\}$; that this definition does not depend on the choice of $\widetilde v$ follows from the $\Gamma$-equivariance of $\varphi$ and $\widetilde f$, \cref{prop:lift-proj-equivariance}, and the fact that $\Gamma$ acts transitively on the fiber $p_G^{-1}(v)$. 
It follows from \cref{prop:lift-proj-pres-adj} (2) that $\Phi(\psi)$ is a set-valued homomorphism $G\pto \Pi H$. 
We have $\Phi(\psi)\in\Hom(G,\Pi H)_f$ by \cref{eqn:f-tilde}.
Thus we obtain a monotone map $\Phi\colon \Hom^\Gamma(\widetilde G,\widetilde H)\to \Hom(G,\Pi H)_f$.

We have $\Phi\Psi=\id_{\Hom(G,\Pi H)_f}$ and $\Psi\Phi=\id_{\Hom^\Gamma(\widetilde G,\widetilde H)}$ since $\lift_{\widetilde f(\widetilde v)}$ and $\proj_{\widetilde f(\widetilde v)}$ are inverse to each other. This completes the proof of \cref{thm:univ-cov-via-tilde}. 

Notice that in the above construction, we have 
\begin{equation}\label{eqn:Psi-maps-idf-to-ftilde}
    \Psi(\id_f)=\widetilde f.
\end{equation}

\subsection{Galois theory of covers}\label{subsec:Galois}
It follows from \cref{thm:univ-cov-via-tilde} and \cref{eqn:Psi-maps-idf-to-ftilde} that the monotone map $p_f\colon E_f\to B_f$ we have constructed in \cref{subsec:def-of-pf} can be equivalently obtained by restricting the monotone map $\beta\colon \Hom^\Gamma(\widetilde G,\widetilde H)\to \Hom(G,H)$ to 
\[
\beta|_{\Conn(\Hom^\Gamma(\widetilde G,\widetilde H),\tilde f)}\colon \Conn\bigl(\Hom^\Gamma(\widetilde G,\widetilde H),\tilde f\bigr)\to \Conn\bigl(\Hom(G,H),f\bigr)=B_f.
\]
In particular, when the graphs $G$ and $H$ satisfy the assumption of \cref{thm:main}, this gives rise to another definition of the universal cover of $B_f$.

It turns out that we can generalize \cref{thm:univ-cov-via-tilde} by replacing the universal cover $p_G\colon \widetilde G\to G$ by a certain (normal) cover of $G$.
Indeed, in the above proof, the only facts about the covering map $p_G\colon \widetilde G\to G$ we have used are
\begin{itemize}
    \item the existence of a $\Gamma$-equivariant graph homomorphism $\widetilde f\colon \widetilde G\to \widetilde H$ making the diagram \cref{eqn:f-tilde} commute, and 
    \item the transitivity of the $\Gamma$-action on each fiber of $p_G$.
\end{itemize}

In order to state the generalization of \cref{thm:univ-cov-via-tilde}, we first recall more about the covering theory of graphs.

Let $G$ be a connected graph and $u\in V(G)$. 
For any subgroup $\Sigma\subseteq \pi_1(G,u)$, we have the equivalence relation $\sim_\Sigma$ on $V(\widetilde G)=V(\Pi G)_u$ defined by 
\[\xi\sim_\Sigma \eta \iff 
t(\xi)=t(\eta)\text{ and } \xi\cdot \eta^{-1}\in \Sigma.\] 
The equivalence class with respect to $\sim_\Sigma$ containing $\xi\in V(\widetilde G)$ is denoted by $[\xi]_\Sigma$.
Define the graph $\widetilde{G}_\Sigma$ by $V(\widetilde{G}_\Sigma)=V(\widetilde G)/_{\sim_\Sigma}$ and 
\[
E(\widetilde{G}_\Sigma)=\bigl\{\,\{[\xi]_\Sigma,[\eta]_\Sigma\}\;\big\vert\; \text{there exist $\xi'\in [\xi]_\Sigma$ and $\eta'\in [\eta]_\Sigma$ such that $\xi'\eta'\in E(\widetilde G)$}\,\bigr\}.
\]
The assignment $\xi\mapsto [\xi]_\Sigma$ defines a surjective graph homomorphism $q_\Sigma\colon \widetilde G\to \widetilde G_\Sigma$. 
We have a graph homomorphism $p_\Sigma\colon \widetilde{G}_\Sigma\to G$ defined by $p_\Sigma([\xi]_\Sigma)=t(\xi)$, i.e., so that the diagram 
\[
\begin{tikzpicture}[baseline=-\the\dimexpr\fontdimen22\textfont2\relax ]
      \node(0) at (0,0.5) {$\widetilde G$};
      \node(1) at (3,0.5) {$\widetilde G_\Sigma$};
      \node(2) at (1.5,-0.5) {$G$};
      \draw [->] (0) to node[auto,labelsize] {$q_\Sigma$} (1);
      \draw [<-] (2) to node[auto,labelsize] {$p_G$} (0);
      \draw [<-] (2) to node[auto,swap,labelsize] {$p_\Sigma$} (1);
\end{tikzpicture}
\]
commutes.
One can show that $p_\Sigma$ is a covering map of graphs. The covering map $p_\Sigma\colon \widetilde{G}_\Sigma\to G$ is said to \emph{correspond to} the subgroup $\Sigma\subseteq \pi_1(G,u)$.
Indeed, one can recover $\Sigma$ from $p_\Sigma\colon \widetilde{G}_\Sigma\to G$ as the image of the (injective) group homomorphism $(p_\Sigma)_\ast\colon \pi_1\bigl(\widetilde{G}_\Sigma,[(u)]_\Sigma\bigr)\to \pi_1(G,u)$.

Let $\Gamma=\pi_1(G,u)$.
If $\Sigma\subseteq \Gamma$ is a normal subgroup of $\Gamma$, the covering map $p_\Sigma\colon \widetilde G_\Sigma\to G$ corresponding to $\Sigma$ is called a \emph{normal} (or \emph{regular}) cover of $G$.
In this case, the natural $\Gamma$-action $\Gamma\times \widetilde G\to \widetilde G$ induces a well-defined $\Gamma$-action $\Gamma\times \widetilde G_\Sigma\to \widetilde G_\Sigma$ sending $(\gamma,[\xi]_\Sigma)\in \Gamma\times V(\widetilde G_\Sigma)$ to $[\gamma\cdot \xi]_\Sigma$; the surjective graph homomorphism $q_\Sigma\colon \widetilde G\to \widetilde G_\Sigma$ is then $\Gamma$-equivariant. 
The kernel of the group homomorphism $\Gamma\to \Aut(\widetilde G_\Sigma)$ corresponding to this $\Gamma$-action on $\widetilde G_\Sigma$ is $\Sigma$.

Now suppose that we are given connected graphs $G$ and $H$, a graph homomorphism $f\colon G\to H$, and a vertex $u\in V(G)$.
Let $\Gamma=\pi_1(G,u)$, and suppose that $\Sigma\subseteq \Gamma$ is a normal subgroup of $\Gamma$ contained in the kernel of the group homomorphism $f_\ast\colon \Gamma=\pi_1(G,u)\to \pi_1\bigl(H,f(u)\bigr)$.
Let $p_G\colon \widetilde G\to G$ be the universal cover of $G$, $p_\Sigma\colon \widetilde G_\Sigma\to G$ the normal cover of $G$ corresponding to $\Sigma$, 
and $p_H\colon \widetilde H\to H$ the universal cover of $H$. 
As explained above, each of $\widetilde G$, $\widetilde G_\Sigma$, and $\widetilde H$ admits a natural $\Gamma$-action. 

We have a monotone map $\Hom^\Gamma(q_\Sigma,\widetilde H)\colon \Hom^\Gamma(\widetilde G_\Sigma,\widetilde H)\to \Hom^\Gamma(\widetilde G,\widetilde H)$ defined by mapping each $\Gamma$-equivariant set-valued homomorphism $\theta\colon \widetilde G_\Sigma\pto \widetilde H$ to the composite 
\[
\begin{tikzpicture}[baseline=-\the\dimexpr\fontdimen22\textfont2\relax ]
      \node(0) at (0,0) {$\widetilde G$};
      \node(1) at (2,0) {$\widetilde G_\Sigma$};
      \node(2) at (4,0) {$\widetilde H$};
      \draw [->] (0) to node[auto,labelsize] {$q_\Sigma$} (1);
      \draw [pto] (1) to node[auto,labelsize] {$\theta$} (2);
\end{tikzpicture}
\]
defined by $(\theta q_\Sigma)(\xi)=\theta([\xi]_\Sigma)$ for each $\xi\in V(\widetilde G)$. We claim that $\Hom^\Gamma(q_\Sigma,\widetilde H)\colon \Hom^\Gamma(\widetilde G_\Sigma,\widetilde H)\to \Hom^\Gamma(\widetilde G,\widetilde H)$ is a poset isomorphism.
Indeed, for any $\Gamma$-equivariant set-valued homomorphism $\psi\colon \widetilde G\pto\widetilde H$, we can define a $\Gamma$-equivariant set-valued homomorphism $\overline \psi\colon \widetilde G_\Sigma\to \widetilde H$ by mapping each $[\xi]_\Sigma\in V(\widetilde G_\Sigma)$ to $\overline \psi([\xi]_\Sigma)=\psi (\xi)$, which is independent of the choice of the representative $\xi\in [\xi]_\Sigma$ by the $\Gamma$-equivariance of $\psi$ and the fact that the $\Sigma$-action on $\widetilde H$ is trivial.

Thus we have shown the following generalization of \cref{thm:univ-cov-via-tilde}.

\begin{corollary}\label{thm:univ-cov-Galois}
    Let $G$ and $H$ be connected graphs, $f\colon G\to H$ a graph homomorphism, and $u\in V(G)$ a vertex.
    Let $\Gamma=\pi_1(G,u)$, and 
    suppose that $\Sigma\subseteq \Gamma$ is a normal subgroup of $\Gamma$ contained in the kernel of the group homomorphism $f_\ast\colon \Gamma\to \pi_1\bigl(H,f(u)\bigr)$. Let $p_\Sigma\colon \widetilde G_\Sigma\to G$ be the normal cover of $G$ corresponding to $\Sigma$
    and $p_H\colon \widetilde H\to H$ be the universal cover of $H$, and make $\widetilde G_\Sigma$ and $\widetilde H$ into $\Gamma$-graphs as above. 
    Then there exists a poset isomorphism $\Psi_\Sigma\colon \Hom(G,\Pi H)_f\to \Hom^\Gamma(\widetilde G_\Sigma,\widetilde H)$ making the diagram 
    \begin{equation*}
\begin{tikzpicture}[baseline=-\the\dimexpr\fontdimen22\textfont2\relax ]
      \node(0) at (0,1) {$\Hom(G,\Pi H)_f$};
      \node(1) at (4,1) {$\Hom^\Gamma(\widetilde G_\Sigma,\widetilde H)$};
      \node(2) at (2,-1) {$\Hom(G,H)$};
      \draw [->] (0) to node[auto,labelsize] {$\Psi_\Sigma$} (1);
      \draw [<-] (2) to node[auto,labelsize] {$\Hom(G,t)|_{\Hom(G,\Pi H)_f}$} (0);
      \draw [<-] (2) to node[auto,swap,labelsize] {$\beta\Hom^\Gamma(q_\Sigma,\widetilde H)$} (1);
\end{tikzpicture}
    \end{equation*}
    commute. 
\end{corollary}

\begin{remark}\label{rmk:change-of-Gamma}
    Under the assumptions and notations of \cref{thm:univ-cov-Galois}, let $\Sigma'$ be a normal subgroup of $\Gamma$ satisfying $\Sigma'\subseteq \Sigma$, and let $\Gamma'$ be the quotient group $\Gamma/\Sigma'$. Then both $\widetilde G_\Sigma$ and $\widetilde H$ can be regarded as $\Gamma'$-graphs, and we can identify the poset $\Hom^\Gamma(\widetilde G_\Sigma,\widetilde H)$ with $\Hom^{\Gamma'}(\widetilde G_\Sigma,\widetilde H)$.
\end{remark}

\begin{example}\label{ex:cover-approach}
    Here are some examples of \cref{thm:univ-cov-Galois}.

    (a) Letting $\Sigma$ be the trivial subgroup of $\Gamma$ consisting only of the unit element, we recover \cref{thm:univ-cov-via-tilde}.

    (b) Another natural choice is to let $\Sigma$ be the kernel of $f_\ast\colon \Gamma\to \pi_1(H,x)$.

    (c) When $H$ is square-free, we can let $\Sigma$ be the normal subgroup of $\Gamma$ whose corresponding quotient group $\Gamma/\Sigma$ is the \emph{$2$-fundamental group} $\pi_1^2(G,u)$ in the sense of \cite[Section~3]{Mat17JMSUT}.
    Taking this choice (and replacing $\Gamma$ by $\Gamma'=\Gamma/\Sigma=\pi_1^2(G,u)$ as in \cref{rmk:change-of-Gamma}) gives rise to Matsushita's approach \cite{Mat25}. The normal cover $p_\Sigma\colon \widetilde G_\Sigma\to G$ is called the \emph{universal $2$-cover} of $G$ in \cite[Section~6]{Mat17JMSUT}.
\end{example}

\newcommand{\etalchar}[1]{$^{#1}$}


\begin{thebibliography}{FIK{\etalchar{+}}25}

\bibitem[Bj{\"o}95]{Bjorner-topology}
Anders Bj{\"o}rner.
\newblock Topological methods.
\newblock In {\em Handbook of combinatorics. Vol. 1-2}, pages 1819--1872. Amsterdam: Elsevier (North-Holland); Cambridge, MA: MIT Press, 1995.

\bibitem[BK06]{BaKo06}
Eric Babson and Dmitry~N. Kozlov.
\newblock Complexes of graph homomorphisms.
\newblock {\em Israel J. Math.}, 152:285--312, 2006.

\bibitem[BM16]{Barmak-Minian-covering}
Jonathan~Ariel Barmak and Elias~Gabriel Minian.
\newblock A note on coverings of posets, {$A$}-spaces and polyhedra.
\newblock {\em Homology Homotopy Appl.}, 18(1):143--150, 2016.

\bibitem[{\v C}K06]{CuKo06}
Sonja~Lj. {\v C}uki\'{c} and Dmitry~N. Kozlov.
\newblock The homotopy type of complexes of graph homomorphisms between cycles.
\newblock {\em Discrete Comput. Geom.}, 36(2):313--329, 2006.

\bibitem[Cso07]{Cso07}
P{\'e}ter Csorba.
\newblock Homotopy types of box complexes.
\newblock {\em Combinatorica}, 27(6):669--682, 2007.

\bibitem[DS12]{DoSc12}
Anton Dochtermann and Carsten Schultz.
\newblock Topology of {H}om complexes and test graphs for bounding chromatic number.
\newblock {\em Israel J. Math.}, 187:371--417, 2012.

\bibitem[FIK{\etalchar{+}}25]{cycle}
Socihiro Fujii, Yuni Iwamasa, Kei Kimura, Yuta Nozaki, and Akira Suzuki.
\newblock Homotopy types of {H}om complexes of graph homomorphisms whose codomains are cycles.
\newblock arXiv:2408.04802v2, 2025.

\bibitem[Hat02]{Hat02}
Allen Hatcher.
\newblock {\em Algebraic topology}.
\newblock Cambridge University Press, Cambridge, 2002.

\bibitem[Kam08]{Kamibeppu}
Akira Kamibeppu.
\newblock Homotopy type of the box complexes of graphs without 4-cycles.
\newblock {\em Tsukuba J. Math.}, 32(2):307--314, 2008.

\bibitem[Koz08]{Koz08}
Dmitry Kozlov.
\newblock {\em Combinatorial algebraic topology}, volume~21 of {\em Algorithms and Computation in Mathematics}.
\newblock Springer, Berlin, 2008.

\bibitem[LMS25]{LMS25}
Benjamin L{\'e}v{\^e}que, Moritz M{\"u}hlenthaler, and Thomas Suzan.
\newblock Reconfiguration of digraph homomorphisms.
\newblock {\em SIAM Journal on Discrete Mathematics}, 39(1):327--360, 2025.

\bibitem[Lov78]{Lov78}
L.~Lov\'{a}sz.
\newblock Kneser's conjecture, chromatic number, and homotopy.
\newblock {\em J. Combin. Theory Ser. A}, 25(3):319--324, 1978.

\bibitem[LS01]{Lyndon-Schupp}
Roger~C. Lyndon and Paul~E. Schupp.
\newblock {\em Combinatorial group theory.}
\newblock Class. Math. Berlin: Springer, reprint of the 1977 ed. edition, 2001.

\bibitem[Mat17]{Mat17JMSUT}
Takahiro Matsushita.
\newblock Fundamental groups of neighborhood complexes.
\newblock {\em J. Math. Sci. Univ. Tokyo}, 24(3):321--353, 2017.

\bibitem[Mat25]{Mat25}
Takahiro Matsushita.
\newblock Hom complexes of graphs whose codomains are square-free.
\newblock arXiv:2412.19144v2, 2025.

\bibitem[ML98]{MacLane-CWM}
Saunders Mac~Lane.
\newblock {\em Categories for the working mathematician}, volume~5 of {\em Grad. Texts Math.}
\newblock New York, NY: Springer, 2nd ed edition, 1998.

\bibitem[MZ04]{MZ-box}
Ji{\v{r}}{\'{\i}} Matou{\v{s}}ek and G{\"u}nter~M. Ziegler.
\newblock Topological lower bounds for the chromatic number: a hierarchy.
\newblock {\em Jahresber. Dtsch. Math.-Ver.}, 106(2):71--90, 2004.

\bibitem[Spa95]{Spanier}
Edwin~H. Spanier.
\newblock {\em Algebraic topology}.
\newblock Berlin: Springer-Verlag, 1995.

\bibitem[Sza09]{Szamuely}
Tam{\'a}s Szamuely.
\newblock {\em Galois groups and fundamental groups}, volume 117 of {\em Camb. Stud. Adv. Math.}
\newblock Cambridge: Cambridge University Press, 2009.

\bibitem[Wac07]{Wachs-poset}
Michelle~L. Wachs.
\newblock Poset topology: tools and applications.
\newblock In {\em Geometric combinatorics}, pages 497--615. Providence, RI: American Mathematical Society (AMS); Princeton, NJ: Institute for Advanced Studies, 2007.

\bibitem[Wal76]{Waller-double-cover}
Derek~A. Waller.
\newblock Double covers of graphs.
\newblock {\em Bull. Aust. Math. Soc.}, 14:233--248, 1976.

\bibitem[Wro20]{Wro20}
Marcin Wrochna.
\newblock Homomorphism reconfiguration via homotopy.
\newblock {\em SIAM J. Discrete Math.}, 34(1):328--350, 2020.

\end{thebibliography}
\end{document}